\documentclass[11pt]{amsart}

\usepackage{a4wide}
\usepackage{amssymb}
\usepackage{amsmath}
\usepackage{epsfig}
\usepackage[sans]{dsfont}

\newcommand{\Subsection}[1]{\subsection{ #1} ${}^{}$}

\newcommand{\one}{\mathds{1}}

\newcommand{\cU}{{\mathcal U}}
\newcommand{\D}{{\mathcal D}}

\newcommand{\R}{{\mathbb R}}

\newcommand{\C}{{\mathbb C}}
\newcommand{\N}{{\mathbb N}}

\newcommand{\rank}{\operatorname{rank}}
\newcommand{\sign}{\operatorname{sign}}

\newcommand{\diag}{\operatorname{diag}}
\newcommand{\half}{\frac{1}{2}}

\newcommand{\mult}{\operatorname{mult}}
\newcommand{\res}{\operatorname{Res}}

\newcommand{\im}{\operatorname{Im}}

\newcommand{\dist}{\operatorname{dist}}

\newcommand{\tr}{{\operatorname{tr}}}
\newcommand{\slim}{\operatornamewithlimits{s--lim}}

\newcommand{\re}{\operatorname{Re}}
\renewcommand{\im}{\operatorname{Im}}

\theoremstyle{plain}

\newtheorem{theorem}{Theorem}
\newtheorem{corollary}[theorem]{Corollary}
\newtheorem{lemma}[theorem]{Lemma}
\newtheorem{proposition}[theorem]{Proposition}
\newtheorem{remark}[theorem]{Remark}

\newtheorem{definition}[theorem]{Definition}

\theoremstyle{definition}

\numberwithin{theorem}{section}
\numberwithin{equation}{section}

\title[Characteristic values and magnetic resonances]{Counting function of characteristic values and magnetic resonances}

\author[J.-F. Bony, V. Bruneau]{Jean-Fran\c{c}ois Bony, Vincent Bruneau}
\address{Institut de Math\'ematiques de Bordeaux, UMR 5251 du CNRS,
Universit\'e de Bordeaux I, 351 cours de la Lib\'eration, 33405 Talence cedex, France}
\email{bony@math.u-bordeaux1.fr, vbruneau@math.u-bordeaux1.fr}
\author[G. Raikov]{Georgi Raikov}
\address{Departamento de Matem\'aticas, Facultad de Matem\'aticas, Pontificia Universidad
Cat\'olica de Chile, Vicu\~na Mackenna 4860, Santiago de Chile}
\email{graikov@mat.puc.cl}



\def\phi {\varphi}

\newcommand{\dom}{{\partial\Omega}}
\newcommand{\cO}{{\mathcal O}}
\newcommand{\cS}{{\mathcal S}}
\newcommand{\cH}{{\mathcal H}}
\newcommand{\cZ}{{\mathcal Z}}
\newcommand{\cN}{{\mathcal N}}
\newcommand{\cK}{{\mathcal K}}
\newcommand{\s}{{\mathcal C}}

\newcommand{\ind}{\operatorname{ind}}
\newcommand{\Ind}{\operatorname{Ind}}

\newcommand{\xp}{X_\perp}
\newcommand{\mper}{m_\perp}

\def\<{\langle}
\def\>{\rangle}

\begin{document}

\begin{abstract}
We consider the meromorphic operator-valued function $I - K(z) = I
- A(z)/z$ where $A$ is holomorphic on the domain ${\mathcal D}
\subset {\mathbb C}$, and has values in the class of compact
operators acting in a given Hilbert space. Under the assumption
that $A(0)$ is  a selfadjoint operator which can be of infinite
rank, we study the distribution near the origin of the
characteristic values of $I - K(z)$, i.e. the complex numbers $w
\neq 0$ for which the operator $I- K(w)$ is not invertible, and
 we show that generically the characteristic values of $I-K$
converge to $0$ with  the same rate as the eigenvalues of
$A(0)$.

We apply our abstract results to the investigation of the
resonances of the operator $H = H_0 + V$ where $H_0$ is the
shifted 3D Schr\"odinger operator with constant magnetic field of
scalar intensity $b>0$, and $V: \R^3 \longrightarrow \R$ is the
electric potential which admits a suitable decay at infinity. It
is well known that the spectrum $\sigma(H_0)$ is purely absolutely
continuous, coincides with $[ 0 , + \infty[$, and the so-called
Landau levels $2bq$ with integer $q \geq 0$, play the role of
thresholds in $\sigma(H_0)$. We study the asymptotic distribution
of the resonances near any given Landau level, and under generic
assumptions obtain the main asymptotic term of the corresponding
resonance counting function, written explicitly in the terms of
 appropriate Toeplitz operators.
\end{abstract}

\maketitle

\section{Introduction}
It is well known that several spectral problems for unbounded
operators can be reduced to the study of a compact-operator-valued function
$K(z)$. Generally, a complex number $z$ in a domain $\D$ is an eigenvalue (or a
resonance) of an operator $H$ if and only if  $I - K(z)$ is not
invertible where $z \longmapsto K(z)$ is holomorphic on  $\D$ with
value in  $\cS_\infty$, the space of compact operators. For example, under suitable assumptions, according to the
Birman--Schwinger principle, the study of the eigenvalues of
$H=H_0+M^*M$ can be related to the compact operator $K(z)= -
M(H_0-z)^{-1}M^*$ (see \cite{Bi61_01}, \cite{BiSo91_01},
\cite{GeHo87_01}, \cite{Kl82_01}, \cite{KlSi80_01},
\cite{Ne83_01}, \cite{Ra80_01}, \cite{Sc61_01}, \cite{Se73_01},
\cite{Si77_01}). For a more general Birman--Schwinger principle
for non-selfajoint operators we refer to \cite{GeLaMiZi05_01}.
Similarly,  the resonances for $H=H_0+V$, a perturbation of a free
Hamiltonian $H_0$, can be analyzed by studying the invertibility
of $I-K(z)$ with $K(z)$ a  compact operator (see \cite{DiZe03_01},
\cite{Fr97_01}, \cite{Fr98_01}, \cite{Gu05_01}, \cite{PeZw01_01},
\cite{SaZw95_01}, \cite{Si00_01}, \cite{Sj97_01}, \cite{Sj01_01},
\cite{Zw89_01}). For example, for perturbations of Schr\"odinger
operators $H_0$ by exponentially decreasing potentials $V$, thanks
to a resolvent equation like \eqref{c7}, we can choose
 $K(z) = - \sign ( V ) \vert V \vert^{\frac{1}{2}} ( H_0 - z )^{-1} \vert V \vert^{\frac{1}{2}}$.
In other situations, $K(z)$ is  constructed by more sophisticated
methods, like Grushin problems or by a representation formula of
the scattering matrix.

In what follows, as in \cite{GoSi71_01}, for $z \longmapsto K(z)$
holomorphic on  $\D$ with  values in  $\cS_\infty$, we will
say that a complex number $w$ is a {\em characteristic value} of
$I - K( \cdot )$ if $I-K(w)$ is not invertible. According to the
analytic Fredholm theorem, if for some $z_0\in \D$ the
operator $I-K(z_0)$ is invertible, then $I-K( \cdot )$ has a discrete set of characteristic values in
$\D$. However, these characteristic values could accumulate at
some point of the boundary $\partial \D$. For example, if $A_0$ is
a selfadjoint compact operator of infinite rank, then the
characteristic values of $I-A_0/z$ in $\C \setminus \{0\}$ are the
eigenvalues of $A_0$ which accumulate at $0$. If $z
\longmapsto K(z)$ is  holomorphic on a domain $\D$, then the
number of characteristic values of $I-K$ in each compact subset of
$\D$ is finite. This property still holds true if $z \longmapsto
K(z)$ is finite meromorphic on $\D$ (see Section \ref{a19},
Proposition \ref{c21} or \cite[Proposition 4.1.4]{GoLe09_01}). For
example, we meet this case within the context of the investigation
of the resonances for the 1D Schr\"odinger operator (see
\cite{Fr97_01}).

In this paper, we consider the case where
\begin{equation} \label{gdr2}
I - K(z) = I - \frac{A (z)}{z},
\end{equation}
with $z \longmapsto A(z):\D \longrightarrow \cS_{\infty}$
holomorphic on a domain  $\D \subset \C$ containing $0$, and
$A(0)$ selfadjoint. As mentioned above, $K(z)$ has typically the
structure of a sandwiched resolvent of the unperturbed operator
which explains its form defined in \eqref{gdr2} where the factor
$1/z$ models, after an appropriate change of the variables, the
threshold singularity of $(H_0-z)^{-1}$.  An important specific
feature of the operators we consider, is  the fact that $A (0)$
can be of infinite rank. Many new phenomena described in the
present article are due to this property. One encounters a similar
situation when one studies the scattering poles on asymptotically hyperbolic manifolds (see \cite{Gu05_01}), or
the resonances of the magnetic Schr\"odinger operator in $\R^3$
(see \cite{BoBrRa07_01}). This type of problem could also arise
for the investigation of resonances near thresholds for other
magnetic Hamiltonians like those of \cite{AsBrBrFeRa08_01},
\cite{Kh09_01}, \cite{Ra10_01}, \cite{Ti11_01}.

First, we consider the asymptotic distribution near the origin of
the characteristic values of $I-A(z)/z$. The natural intuition is
that these characteristic values  accumulate at $0$ with
the same rate as the spectrum of $A(0)$,  but since only $A (0)$ is assumed to be selfadjoint, some pseudospectral phenomena could perturb this conjecture. We describe situations where this intuition is really valid (see Section \ref{s2}). Then we apply
our abstract results to the study of the distribution of
resonances near the spectral thresholds for the shifted 3D
Schr\"odinger operators with constant magnetic field of strength
$b>0$, pointing at the $x_3$-direction:
\begin{equation} \label{gdr0}
H(b, V): = \Big( D_{1} + \frac{b}{2} x_2 \Big)^2 + \Big( D_{2} - \frac{b}{2} x_1 \Big)^2 - b + D_{3}^2 + V , \qquad D_j := -i \frac{\partial}{\partial x_j} .
\end{equation}
We regard this operator as one of the main sources
of motivation for the article, and hence we would like to discuss
it in more detail. Set $\xp = (x_1, x_2) \in \R^2$. Using the
representation $L^2(\R^3) = L^2(\R_{\xp}^2) \otimes
L^2(\R_{x_3})$, we find that
\begin{equation} \label{gdr10}
H_0 : = H(b,0) = H_{\rm Landau} \otimes I_3 + I_{\perp} \otimes
\Big( - \frac{\partial^2}{\partial x_{3}^2} \Big)
\end{equation}
where
\begin{equation} \label{gdr4}
 H_{\text{Landau}}: = \Big( D_1 + \frac{b}{2} x_2 \Big)^2 + \Big( D_2 - \frac{b}{2} x_1 \Big)^2 -b,
\end{equation}
is the shifted Landau Hamiltonian, selfadjoint in $L^2(\R^2)$,
and $I_3$ and $I_\perp$ are the identity operators in
$L^2(\R_{x_3})$ and $L^2(\R_{\xp}^2)$ respectively. It is well
known that the spectrum of $H_{\text{Landau}}$ consists of the
so-called Landau levels $2bq$,  $q \in \N : = \{0,1,2, \ldots\}$,
and ${\rm dim}\,{\rm Ker}(H_{\text{Landau}} - 2bq) = \infty$. Consequently,
\begin{equation*}
\sigma ( H_0 ) = \sigma_{\rm ac} ( H_0 ) = [ 0 , + \infty [ ,
\end{equation*}
and we conclude that the Landau levels play the role of thresholds
in the spectrum of $H_0$. Since the ``transversal'' operator
$H_{\text{Landau}}$ in \eqref{gdr10} has a purely point spectrum,
and its eigenvalues form a discrete subset of $\R$ while the
spectrum of the ``longitudinal'' operator $-
\frac{\partial^2}{\partial x_3^2}$ is purely absolutely
continuous, the structure of $H_0$ is quite close to the one of
the (unperturbed) quantum waveguide Hamiltonians. The study of the
resonances for perturbations of such quantum waveguides and their
generalizations has a rich history (see e.g. \cite{AsPaVa00_01}, \cite{Ch04_01}, \cite{Ed02_01},
\cite{WuZw00_01}). The novelty of
the results obtained in the present article as well as in its
predecessor \cite{BoBrRa07_01} is related to the fact that in the
case of the operator $H(b,0)$ the spectral thresholds (i.e. the
eigenvalues of the transversal operator $H_{\text{Landau}}$ in
\eqref{gdr10}) are of infinite multiplicity; this corresponds to
the fact that $\rank A(0) = \infty$ in the case of the
operator in \eqref{gdr2} (see below \eqref{c10} for the explicit
expression of the operator $A$ arising in the study of the
resonances accumulating at the $q$th Landau level).

Assume now that the multiplier by the electric potential $V : \R^3
\longrightarrow \R$ is relatively compact with respect to $H_0$. It is known
that if $V$ satisfies the estimate
\begin{equation} \label{gdr1}
V({\bf x}) \leq - C \one_{U}({\bf x}), \qquad {\bf x} \in \R^3,
\end{equation}
where $C>0$ and  $U \subset \R^3$ is an open non empty set, the operator $H(b, V)$ has an
infinite negative discrete spectrum (see e.g. \cite[Theorem
1.5]{AvHeSi78_01}). Next, if $V$ is axisymmetric, i.e. depends
only on $|\xp|$ and $x_3$, and satisfies \eqref{gdr1}, then below
each Landau level $2bq$, $q \in \N$, the operator $H(b, V)$ has at
least one eigenvalue which for all sufficiently large $q$ is
embedded in the essential spectrum (see  \cite[Theorem
1.5]{AvHeSi78_01}). Finally, if $V$ is axisymmetric and satisfies
\begin{equation} \label{y1}
V({\bf x}) \leq - C \one_{W}(\xp) (1+|x_3|)^{-m_3}, \qquad {\bf x}
= (\xp, x_3) \in \R^3,
\end{equation}
where $C>0$, $m_3 \in (0,2)$ and  $W \subset \R^2$ is an open non empty set, then there
exists an infinite series of eigenvalues of $H(b, V)$ below each
Landau levels $2bq$ (see \cite{Ra05_01}, \cite{Ra06_01}). Further, in \cite{FeRa04_01} it was supposed that $V$ is continuous, has a definite sign, does not vanish identically, and satisfies
\begin{equation} \label{t10}
V({\bf x}) = {\mathcal O} \big( (1 + \vert {\xp}
\vert )^{-\mper} (1+|x_3|)^{-m_3} \big) , \qquad {\bf x}= (\xp,x_3) \in \R^3,
\end{equation}
and it was shown that the Krein spectral shift function associated
with the operator pair $(H(b, V),H(b,0))$ has singularities
at the Landau levels. All these properties suggest that for
generic $V$ there could be an accumulation of resonances of $H(b,
V)$ at the Landau levels.

In the present article we assume that  $V$ is Lebesgue measurable,
and satisfies
\begin{equation}  \label{s21}
V({\bf x}) = {\mathcal O} \big(  (1 + \vert {\xp} \vert )^{-\mper}
\exp(-N \vert x_3 \vert ) \big) , \qquad {\bf x}= (\xp,x_3) \in
\R^3,
\end{equation}
with  $\mper >0$,  $N>0$.  In \cite{BoBrRa07_01},  we defined the
resonances of $H(b,V)$ under this assumption, and gave an upper
bound on their number at a distance $r\searrow 0$ of the Landau
levels. Moreover, for $H ( b , e V )$ with $e$ sufficiently small
and $V$ of definite sign, compactly supported (or decreasing like
a Gaussian function), we stated a lower bound. Here, in Section
\ref{s5}, we obtain the asymptotic
 behavior of the counting function of the magnetic resonances near the
 Landau levels for $V$ of definite sign satisfying the estimate \eqref{s21} and
 for every $e\in \R\setminus {\mathcal E}$ where ${\mathcal E}$ is a
 discrete set of $\R$.
To our best knowledge, it is the first result giving such
asymptotic behaviour for counting functions of magnetic resonances
(and maybe, more generally for resonances which accumulate
at a spectral threshold).

Apart from the applications in the mathematical theory of
resonances for quantum Hamiltonians, we hope that our research
could turn out to be useful for the better understanding of the so
called magnetic Feshbach resonances which play an important role
in the modern theoretical physics, in particular the theory of
Bose--Einstein condensates (see e.g. \cite{BaLeVoReDu09_01},
\cite{PaKrCoSa05_01}, \cite{SuKr11_01}).

The article is organized as follows. In Section \ref{a19} we
introduce the notions of characteristic values and  index of an
operator with respect to a contour, following mainly
\cite{GoLe09_01} and \cite{GoSi71_01}. These tools are used
throughout the article; note in particular that they avoid the use
of the regularized determinants. Further, we  study the asymptotic
distribution as $r \searrow 0$ of the characteristic values in a
domain of size $r$, situated  at a distance $r$ from the origin
(see Theorem \ref{a17}). Then, we prove a general result
concerning the asymptotics as $r \searrow 0$ of the characteristic
values in  a domain of size $1$ situated at a distance $r$ from
the origin  (see Theorem \ref{a17b}, Corollary \ref{a17c} and
Corollary \ref{a17d}). These abstract results are stated in
Section \ref{s2}, and are proved respectively  in Section \ref{s3}
and in Section \ref{s4}. In Section \ref{s5}, we apply our
abstract results to magnetic Schr\"odinger operators.
 Eventually, in Section \ref{s7}, we construct  some counterexamples
which show that the assumptions of the results of Section \ref{s2}
can not be removed.

\section{Characteristic values of holomorphic operators} \label{a19}

In this section, we define the notions of characteristic values of
an operator valued holomorphic function and their multiplicities.
For more details, we refer to \cite{GoSi71_01} and to Section 4 of
\cite{GoLe09_01}.

For the formulation of our results we need the following notations
used throughout the article. Let $\cH$ be a separable Hilbert
space. We denote by ${\mathcal L}(\cH)$ (resp. $\cS_\infty(\cH)$)
the class of linear bounded (resp. compact) operators acting in
$\cH$. By ${\mathcal G}{\mathcal L}(\cH)$, we denote the class of
invertible bounded operators, and by $\cS_p(\cH)$, $p \in [ 1 , +
\infty [$, the Schatten--von Neumann classes of compact operators.
In particular $\cS_1$ is the trace class, and $\cS_2$ is the
Hilbert--Schmidt class. When appropriate, we omit the explicit
indication of the Hilbert space $\cH$ where the operators
from a given class act.

\begin{definition}\sl \label{sl}
For $w \in \C$, let ${\cU}$ be a neighborhood of $w$, and let $F :
\cU\setminus \{ w\} \longrightarrow {\mathcal L}(\cH)$ be a
holomorphic function. We say that $F$ is finite meromorphic at $w$
if the Laurent expansion of $F$ at $w$ has the form
\begin{equation*}
F(z)= \sum_{n=m}^{+ \infty} (z-w)^n A_n, \qquad m > -\infty,
\end{equation*}
the operators $A_m, \cdots , A_{-1}$ being of finite rank, if $m<0$.

If, in addition, $A_0$ is a Fredholm operator, then $F$ is called
Fredholm at $w$, and the  Fredholm index of $A_0$ is called the
Fredholm index of $F$ at $w$.
\end{definition}

\begin{remark}\sl \label{grr1}
Definition \ref{sl}, as well as most of the results of the present
section, admits a generalization to a Banach-space setting. We
formulate these results in a form sufficient for our purposes.
\end{remark}

\begin{proposition}[{\cite[Proposition 4.1.4]{GoLe09_01}}]\sl \label{c21}
Let $\D \subset \C$ be a connected open set, let $Z\subset \D$ be
a discrete and closed subset of $\D$, and let $F : \D
\longrightarrow {\mathcal L} (\cH)$ be a holomorphic function on
$\D\setminus Z$. Assume that:
\begin{itemize}
 \item $F$ is finite meromorphic on $\D$, i.e. it is finite meromorphic in a vicinity of each point of $Z$;
 \item $F$ is Fredholm at each point of $\D$;
 \item there exists $z_0 \in \D\setminus Z$ such that $F(z_0)$ is invertible.
 \end{itemize}
Then there exists a discrete and closed subset $Z^\prime$ of $\D$ such that:
\begin{itemize}
\item $Z \subset Z^\prime$;
\item $F(z)$ is invertible for $z \in \D\setminus Z^\prime$;
\item $F^{-1} : \D\setminus Z^\prime \longrightarrow {\mathcal G}{\mathcal L}(\cH)$ is finite meromorphic and Fredholm at each point of $\D$.
\end{itemize}
\end{proposition}

Then we can define the {\em characteristic values}
of $F$, and their multiplicities.

\begin{definition}\sl \label{c16}
In the setting of Proposition \ref{c21}, each point of
$Z^\prime$, where $F$ or $F^{-1}$ is not holomorphic, is called a
characteristic value of $F$. The multiplicity of a characteristic
value $w_0$ is defined by
\begin{equation}\label{c17}
\mult (w_0):= \frac{1}{2 i \pi} \tr \int_{\vert w - w_0\vert=\rho} F^{\prime} (z) F (z)^{-1} d z,
\end{equation}
where $\rho>0$ is sufficiently small such that $\{ w ; \ \vert w - w_0\vert
\leq \rho \} \cap Z^\prime = \{ w_0 \}$.
\end{definition}

By definition, if $F$ is holomorphic in $\D$, a characteristic
value of $F$ is a complex number $w$ for which $F(w)$ is not
invertible. Then, according to results of \cite{GoSi71_01}
and \cite[Section 4]{GoLe09_01}, $\mult (w)$ is an integer.
Moreover, the definition of the multiplicity coincides with a
definition of the order of $w$ as a zero of $F$ (see
\cite{GoSi71_01} for more details).

If $\Omega \subset \D$ is a connected domain such that $\partial
\Omega \cap Z^\prime = \emptyset$, then the sum of the
multiplicities of the characteristic values of $F$ inside $\Omega$
is the so-called {\em index of $F$ with respect to the contour
$\partial \Omega$}, given by
\begin{equation}\label{c20}
\Ind_{\partial \Omega} F: = \frac{1}{2 i \pi} \tr \int_{\partial \Omega} F^{\prime} (z) F (z)^{-1} d z = \frac{1}{2 i \pi} \tr \int_{\partial \Omega} F (z)^{-1} F^{\prime} (z) \, d z .
\end{equation}

We easily check  that
\begin{equation} \label{gr4}
\Ind_{\partial \Omega}( F_1F_2)= \Ind_{\partial \Omega} F_1 +
\Ind_{\partial \Omega} F_2 ,
\end{equation}
 provided that the
operator-valued functions $F_1$ and $F_2$ satisfy the assumptions of
Proposition \ref{c21}. Let us remark also that if $I-F \in \cS_1$, then $F^{\prime} (z) F (z)^{-1} \in \cS_1$ for $z \in \D\setminus
Z^\prime$, and we have
\begin{equation} \label{gr5}
\Ind_{\partial \Omega} F = \frac{1}{2 i \pi} \int_{\partial \Omega} \tr \big( F^{\prime} (z) F (z)^{-1} \big) \, d z = \ind_{\partial \Omega} f,
\end{equation}
where $f(z) = \det ( F ( z ) )$ is the Fredholm determinant of $F(z)$ and $\ind_{\partial \Omega} f$ is the standard index of a holomorphic function equal to the number of its zeroes in $\Omega$:
\begin{equation} \label{d1}
\ind_{\partial \Omega} f:= \frac{1}{2 i \pi} \int_{\partial \Omega} \frac{f^{\prime} (z)}{ f (z)} \, d z.
\end{equation}

More generally, if $I-F \in \cS_p$ with integer $p \geq 2$, then
the regularized determinant of $F$ is well defined, namely
\begin{equation*}
f_p(z)={\det}_{p}(F(z)):=\det \Big( F(z)
\exp \Big( \sum_{k=1}^{p-1} \frac{1}{k}(I - F (z))^k \Big) \Big) ,
\end{equation*}
(see \cite{Ko84_01}, \cite{Ko85_01}, \cite{Kr62_01}), and if
$\partial \Omega \cap Z^\prime = \emptyset$, we have
\begin{equation}\label{c9}
\ind_{\partial \Omega} f_p= \frac{1}{2 i \pi} \int_{\partial
 \Omega}\tr\Big( F^{\prime} (z) F (z)^{-1} -
\sum_{k=1}^{p-1} F^{\prime} (z)(I-F(z))^{k-1} \Big) \, d z=
\Ind_{\partial \Omega} F,
\end{equation}
 since the sum in the above formula is the derivative of a
function.

Moreover, we have a Rouch\'e-type theorem:

\begin{theorem}[{\cite[Theorem 4.4.3]{GoLe09_01}, \cite[Theorem 2.2]{GoSi71_01}}]\sl \label{a18}
For $\D \subset \C$ a bounded open set with piecewise $C^1$-boundary and $Z \subset \D$ a finite set,
let $F: \overline{\D} \setminus Z \longrightarrow {\mathcal G}{\mathcal L} (\cH)$ be a holomorphic function
which is finite meromorphic and Fredholm at each point of $Z$ and let $G: \overline{\D} \setminus Z \longrightarrow
{\mathcal L} (\cH)$ be a holomorphic function which is finite meromorphic at each point of $Z$, and satisfies
\begin{equation*}
\Vert F (z)^{-1} G(z)\Vert < 1, \qquad z \in \partial \D.
\end{equation*}
Then $F+G$ is finite meromorphic and Fredholm at each point of
$Z$, and
\begin{equation*}
\Ind_{\partial \D} (F+G) = \Ind_{\partial \D} (F).
\end{equation*}
\end{theorem}

\section{Asymptotic expansions: abstract results}\label{s2}

Let $\D$ be a domain of $\C$ containing $0$, and $\cH$ be a separable Hilbert space. We consider a holomorphic operator function
\begin{equation*}
A: \D \longrightarrow \cS_{\infty} ( \cH ).
\end{equation*}
For $\Omega \subset \D$, we denote by $\cZ ( \Omega )$ the
set of the characteristic values of $I- \frac{A (z)}{z}$,
i.e.
\begin{equation*}
\cZ ( \Omega ) :=  \Big\{ z \in \Omega \setminus \{ 0 \} ; \
I-\frac{A (z)}{z} \text{ is not invertible} \Big\} ,
\end{equation*}
and by $\cN ( \Omega )$ the number of characteristic values in
$\Omega$ counted with their multiplicities, i.e.
\begin{equation*}
\cN ( \Omega ):= \# \cZ ( \Omega ) .
\end{equation*}
We refer to Section \ref{a19} for details concerning the characteristic values.

We deduce from Proposition \ref{c21} that $\cZ ( \D )$ is a finite
set in a neighborhood of the origin as soon as $A(0)$ is of finite
rank. In this section we assume that $A(0)$ is a selfadjoint
operator and we are mainly interested in the case where $A(0)$ is
of infinite rank.

We will formulate results concerning the number of characteristic
values of $I - \frac{A (z)}{z}$ in two  types of domains:
small domains  of the form $s\Omega$ with $\Omega \Subset \C
\setminus \{ 0 \}$ fixed and $s$ tending to $0$, and sectorial
domains of the form
\begin{equation}\label{z1}
{\s}_{\theta}(a,b) : = \{ x + i y \in \C; \ a \leq x \leq b , \ \vert y \vert  \leq \theta \vert x \vert \} ,
\end{equation}
with $b , \theta >0$ fixed and $a>0$ tending to $0$.

Our goal is to describe situations where  the
behaviour of $\cN(s\Omega)$ as $s \searrow 0$, or of
$\cN(\s_{\theta} (r,1))$ as $r \searrow 0$, is related to the
asymptotics of the number
\begin{equation*}
n ( \Lambda ) : = \tr \one_{\Lambda} ( A (0) ),
\end{equation*}
of the eigenvalues of the operator $A (0)$, lying in an
appropriate set $\Lambda \subset \R$, and
 counted with their multiplicities. Let $\Pi_{0}$ be the orthogonal projection  onto $\ker A (0)$, and
$\overline{\Pi}_{0} := I - \Pi_{0}$. In small domains we have:

\begin{theorem}\sl \label{a17}
Let $\D$ be a domain of $\C$ containing $0$ and let $A$ be a holomorphic operator-valued function
\begin{equation*}
A: \D \longrightarrow \cS_{\infty} ( \cH ) ,
\end{equation*}
such that $A(0)$ is selfadjoint and $I - A^{\prime} (0) \Pi_{0}$
is invertible. Assume that $\Omega \Subset \C \setminus \{ 0 \}$
is a bounded domain with smooth boundary $\partial \Omega$ which
is transverse to the real axis at each point of $\partial \Omega
\cap \R$. Then, for all $\delta > 0$ small enough, there exists $s
( \delta  ) > 0$ such that, for all $0 < s < s ( \delta )$, we
have
\begin{equation*}
\cN (s \Omega ) = n (s J ) + \cO \big( n (s I_{\delta} ) \vert \ln \delta \vert^{2} \big) ,
\end{equation*}
where $J : = \Omega \cap \R$, $I_{\delta} : = \partial \Omega \cap
\R + [- \delta , \delta ]$ and the $\cO$ is uniform with respect to $s , \delta$.
\end{theorem}

\begin{figure}
\begin{center}
\begin{picture}(0,0)%
\includegraphics{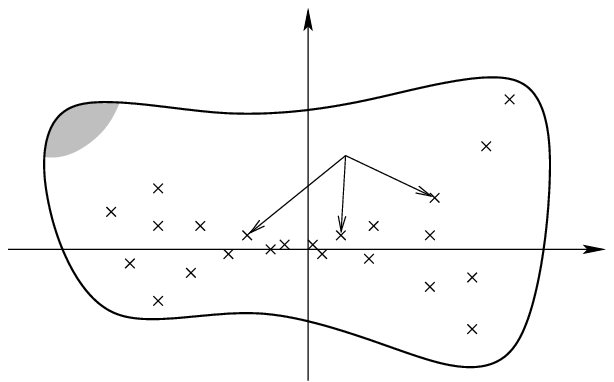}%
\end{picture}%
\setlength{\unitlength}{1184sp}%
\begingroup\makeatletter\ifx\SetFigFont\undefined%
\gdef\SetFigFont#1#2#3#4#5{%
  \reset@font\fontsize{#1}{#2pt}%
  \fontfamily{#3}\fontseries{#4}\fontshape{#5}%
  \selectfont}%
\fi\endgroup%
\begin{picture}(9926,6044)(2097,-7283)
\put(7501,-3511){\makebox(0,0)[lb]{\smash{{\SetFigFont{9}{10.8}{\rmdefault}{\mddefault}{\updefault}$\cZ ( \D )$}}}}
\put(3676,-4036){\makebox(0,0)[lb]{\smash{{\SetFigFont{9}{10.8}{\rmdefault}{\mddefault}{\updefault}$\D$}}}}
\put(7726,-2011){\makebox(0,0)[lb]{\smash{{\SetFigFont{9}{10.8}{\rmdefault}{\mddefault}{\updefault}$\C$}}}}
\end{picture}%
\qquad
\begin{picture}(0,0)%
\includegraphics{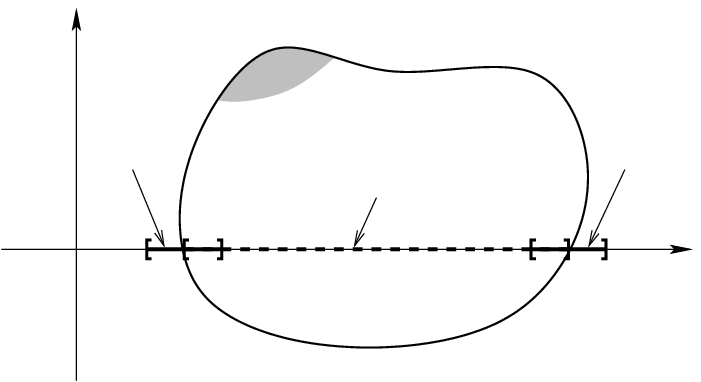}%
\end{picture}%
\setlength{\unitlength}{1184sp}%
\begingroup\makeatletter\ifx\SetFigFont\undefined%
\gdef\SetFigFont#1#2#3#4#5{%
  \reset@font\fontsize{#1}{#2pt}%
  \fontfamily{#3}\fontseries{#4}\fontshape{#5}%
  \selectfont}%
\fi\endgroup%
\begin{picture}(11144,6044)(1779,-7283)
\put(6301,-3211){\makebox(0,0)[lb]{\smash{{\SetFigFont{9}{10.8}{\rmdefault}{\mddefault}{\updefault}$s \Omega$}}}}
\put(3526,-2011){\makebox(0,0)[lb]{\smash{{\SetFigFont{9}{10.8}{\rmdefault}{\mddefault}{\updefault}$\C$}}}}
\put(7651,-4186){\makebox(0,0)[lb]{\smash{{\SetFigFont{9}{10.8}{\rmdefault}{\mddefault}{\updefault}$s J$}}}}
\put(11626,-3736){\makebox(0,0)[lb]{\smash{{\SetFigFont{9}{10.8}{\rmdefault}{\mddefault}{\updefault}$s I_{\delta}$}}}}
\put(3601,-3736){\makebox(0,0)[lb]{\smash{{\SetFigFont{9}{10.8}{\rmdefault}{\mddefault}{\updefault}$s I_{\delta}$}}}}
\put(10351,-4861){\makebox(0,0)[lb]{\smash{{\SetFigFont{9}{10.8}{\rmdefault}{\mddefault}{\updefault}$s \delta$}}}}
\end{picture}%
\caption{The set of characteristic values $\cZ ( \D )$ and the setting of Theorem \ref{a17}.} \label{f1}
\end{center}
\end{figure}

\begin{remark}\sl \label{c2a}
In the context of Theorem \ref{a17}, the assumptions of
Proposition \ref{c21} hold true and then the characteristic values
are well defined. This follows from the hypotheses of Theorem
\ref{a17} and Proposition \ref{p35} below.
\end{remark}

\begin{remark}\sl \label{c2}
In Theorem \ref{a17}, the remainder estimate is uniform with
respect to some perturbations of the domain $\Omega$. For example,
let $\theta > 0$ and $0 < a_{-} \leq a_{+} < b_{-} \leq b_{+} < +
\infty$. Then the conclusion of Theorem \ref{a17} holds for
$\Omega = \s_{\theta} (a ,b)$ uniformly with respect to $a,b$ such
that $a_{-} \leq a \leq a_{+}$ and $b_{-} \leq b \leq b_{+}$.
\end{remark}

The setting is illustrated on Figure \ref{f1}. With respect to
different types of $\Omega$,  Theorem \ref{a17} implies the
following properties on the characteristic values near $0$:

\begin{corollary}\sl \label{a22}
Under the assumptions of Theorem \ref{a17}, we have

$i)$ If $\Omega \cap \R = \emptyset$ then $\cN ( s \Omega ) = 0$ for $s$ small enough. This implies that the characteristic values $z \in \cZ ( \D )$ near $0$ satisfy
\begin{equation*}
\vert \im z \vert = o ( \vert z \vert ) .
\end{equation*}

$ii)$ Moreover, if $A (0)$ has a  definite sign, i.e. $\pm A
(0) \geq 0$, then the characteristic values $z$ near $0$ satisfy
\begin{equation*}
\pm \re z \geq 0 .
\end{equation*}

$iii)$ If $A (0)$ is of finite rank, then there are no
characteristic values in a pointed neighborhood of $0$. Moreover,
if $A (0) \one_{ [0 , + \infty [} (\pm A (0) )$ is of
finite rank, then there are no characteristic values in a
neighborhood of $0$ intersected with $\{ \pm \re z > 0 \}$.
\end{corollary}

\begin{remark}\sl
According to Proposition \ref{c21}, the first sentence of
Corollary \ref{a22} $iii)$ holds true even if $A(0)$ is non
selfadjoint.
\end{remark}

In fact Corollary \ref{a22} is a consequence of the following
result which, under an appropriate spectral condition on $A(0)$,
guarantees the existence of a region free of  characteristic values of $I - A(z)/z$.  We omit its proof since it is quite similar to the one of Lemma \ref{a4} below.

\begin{proposition}\label{p35}\sl
Assume the hypotheses of Theorem \ref{a17}. Let  the operator-valued function
\begin{equation*}
z \longmapsto  \Big( I - \frac{A(0)}{z}\Big)^{-1} ,
\end{equation*}
be well defined and uniformly bounded on the set $S \subset \D \setminus \{0\}$. Then, there exist $r_0>0$ such that   $\cZ ( S ) \cap \{ \vert z \vert < r_0 \} = \emptyset$, and a constant $C>0$ independent of  $r_0$ and $S$, such that
\begin{equation*}
\sup_{z \in S \cap \{ \vert z \vert < r_{0} \}} \Big\Vert \Big( I-\frac{A(z)}{z}\Big)^{-1} \Big\Vert \leq C \sup_{z \in S \cap \{ \vert z \vert < r_{0} \}}  \Big\Vert \Big( I-\frac{A(0)}{z}\Big)^{-1} \Big\Vert ,
\end{equation*}
if $S \cap \{ \vert z \vert < r_0 \} \neq \emptyset$.
\end{proposition}

Exploiting Theorem \ref{a17} with appropriate domains $\Omega$, we obtain the following result in sectorial domains:

\begin{theorem}\sl \label{a17b}
Let $\D$ be a domain of $\C$ containing $0$ and let a holomorphic operator function
\begin{equation*}
A: \D \longrightarrow \cS_{\infty} ( \cH ) ,
\end{equation*}
such that $A(0)$ is selfadjoint and $I - A^{\prime} (0) \Pi_{0}$
is invertible. For $\theta > 0$ fixed, let $\s_{\theta} (r,1)
\subset \D$ be defined as in \eqref{z1}. Then, for all
$\delta
> 0$ small enough, there exists $s ( \delta  ) > 0$ such that, for
all $0 < s < s ( \delta )$, we have
\begin{equation*}
\cN ( \s_{\theta} (r,1) )= n ([r,1] ) \big( 1 + \cO \big( {\delta} \vert \ln \delta \vert^{2} \big) \big) + \cO \big( \vert \ln \delta \vert^{2} \big) n ( [r(1-\delta),r(1+\delta)] ) + \cO_{\delta} (1) ,
\end{equation*}
where the $\cO$'s are uniform with respect to $s , \delta$ but the $\cO_{\delta}$ may depend on $\delta$.
\end{theorem}

\begin{figure}
\begin{center}
\begin{picture}(0,0)%
\includegraphics{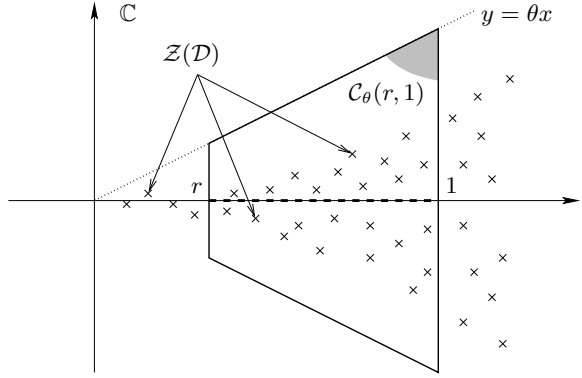}%
\end{picture}%
\setlength{\unitlength}{1184sp}%
\begingroup\makeatletter\ifx\SetFigFont\undefined%
\gdef\SetFigFont#1#2#3#4#5{%
  \reset@font\fontsize{#1}{#2pt}%
  \fontfamily{#3}\fontseries{#4}\fontshape{#5}%
  \selectfont}%
\fi\endgroup%
\begin{picture}(12044,7855)(1179,-8794)
\put(11101,-1411){\makebox(0,0)[lb]{\smash{{\SetFigFont{9}{10.8}{\rmdefault}{\mddefault}{\updefault}$y = \theta x$}}}}
\put(3526,-1411){\makebox(0,0)[lb]{\smash{{\SetFigFont{9}{10.8}{\rmdefault}{\mddefault}{\updefault}$\C$}}}}
\put(10351,-5011){\makebox(0,0)[lb]{\smash{{\SetFigFont{9}{10.8}{\rmdefault}{\mddefault}{\updefault}$1$}}}}
\put(5026,-5011){\makebox(0,0)[lb]{\smash{{\SetFigFont{9}{10.8}{\rmdefault}{\mddefault}{\updefault}$r$}}}}
\put(8326,-3061){\makebox(0,0)[lb]{\smash{{\SetFigFont{9}{10.8}{\rmdefault}{\mddefault}{\updefault}$\mathcal{C}_{\theta} (r,1)$}}}}
\put(4426,-2236){\makebox(0,0)[lb]{\smash{{\SetFigFont{9}{10.8}{\rmdefault}{\mddefault}{\updefault}$\mathcal{Z} ( \mathcal{D} )$}}}}
\end{picture}%
\caption{The setting of Theorem \ref{a17b} and its corollaries.} \label{f4}
\end{center}
\end{figure}

\begin{remark}\sl
$i)$ The same result holds mutatis mutandis for $\s_{\theta} ( - 1 , - r  )$, $r>0$, under the same assumptions. To prove this point, it is sufficient to change the variable  $z \to -z$ and replace $A(z)$ by $-A(-z)$.

$ii)$ The dependence on $\theta$ in Theorem \ref{a17b} is inessential. More precisely, thanks to Corollary \ref{a22} $i)$,  the difference $\cN ( \s_{\theta} (r,1) ) - \cN ( \s_{\nu} (r,1) )$ is constant for $r$ small enough.
\end{remark}

Now, let us give conditions on $n ([r,1] )$ for which Theorem \ref{a17b} implies asymptotic expansion of
$\cN (\s_{\theta} (r,1  ) )$ with main term $n ([r,1] )$.

\begin{corollary}\sl \label{a17c}
 Let the assumptions of Theorem \ref{a17b} hold true.  Suppose that there exists $\gamma>0$ such that
\begin{equation} \label{c23}
n ([r,1] )= \cO(r^{-\gamma}), \qquad  r \searrow 0,
\end{equation}
and that $n ([r,1] )$  grows unboundedly as $r \searrow 0$. Then there exists a  positive sequence $(r_k)_{k \in \N}$ tending to $0$,  which satisfies
\begin{equation} \label{end4}
\cN (\s_{\theta} (r_k,1 ) ) = n ([r_k,1] ) (1 + o(1)), \qquad  k \to + \infty .
\end{equation}
\end{corollary}

\begin{remark}\sl
Note that  \eqref{c23} is equivalent to the fact that
 $A (0) \one_{[0, +\infty [} (A (0) ) \in
\cS_{p}$ for some $p \in [ 1 , + \infty [$. Such hypothesis is
reasonable since it guarantees that, in some sense, $A (0)$ has
less eigenvalues  in a vicinity of $r$ than in the whole interval $[r , 1]$ (see Lemma \ref{b6} below).
 On the other hand, the assumption that $n ([r,1] )$  grows unboundedly as $r \searrow 0$ is equivalent to ${\rm rank}\, A (0) \one_{[0 , + \infty [} (A (0) ) = +\infty$.
\end{remark}

\begin{corollary}\sl \label{a17d}
Let the assumptions of Theorem \ref{a17b} hold true. Suppose that
\begin{equation*}
n ([r,1] )= \Phi (r) (1+ o(1)) , \qquad r \searrow 0 ,
\end{equation*}
with $\Phi(r) = C r^{-\gamma}$, or $\Phi(r) = C \vert \ln r \vert^\gamma$, or  $\Phi(r) = C \frac{\vert \ln r \vert}{\ln \vert \ln r \vert}$, for some $\gamma ,C > 0$. Then
\begin{equation*}
\cN (\s_{\theta} (r,1 ) ) = \Phi (r) (1 + o(1)), \qquad  r \searrow 0.
\end{equation*}
\end{corollary}

These results are proved in Section \ref{s4}.

\begin{remark}\sl
As shown in Section \ref{s7}, the assumptions of Theorem \ref{a17} and Theorem \ref{a17b} (compactness of $A (z)$, selfadjointness of $A(0)$, and invertibility of $I-A^{\prime}(0)\Pi_0$) are necessary in order to have the claimed results.
\end{remark}

\section{Proof of Theorem \ref{a17}}\label{s3}

Throughout the section we assume that the hypotheses of Theorem \ref{a17} are fulfilled. Note that it is enough to prove this result for $\delta$ small enough. In the following, $C$ will denote a positive constant independent of $\delta , s , z$ which may change its value from line to line.

\Subsection{Estimate of the resolvents}

For $0 < \delta \leq \delta_{0} : = \min ( 1 , \dist (
\partial \Omega \cap \R , \{ 0 \} ) / 2 )$ and $s > 0$, let us
introduce the  finite rank operator
\begin{equation} \label{a21}
K_{\delta} (s) : = A(0) \one_{I_\delta} \Big( \frac{A(0)}{s} \Big) ,
\end{equation}
where $I_{\delta} = \partial \Omega \cap \R + [- \delta , \delta ]$. This operator will be used to ``remove'' the eigenvalues of $A (0)$ which are at distance $\delta s$ from $\partial \Omega \cap \R$. We also define the set
\begin{equation*}
V_{\nu , \delta} := ( \Omega + B (0 , \nu ) ) \cap \{ \vert \im z
\vert > \varepsilon \delta \} \qquad \text{ and } \qquad W_{\delta} := ( \partial \Omega \cap \R ) + B (0 , \delta ) ,
\end{equation*}
where $\nu >0$ is a constant chosen sufficiently small so that $\Omega + B (0 , \nu ) \Subset \C \setminus \{ 0 \}$ and $\varepsilon \in ] 0 ,1/2 [$ will be chosen later so that $\partial \Omega \subset V_{\nu /2 , 2 \delta} \cup W_{\delta /4}$ . For $\delta > 0$ small enough, the setting is illustrated  in Figure \ref{f2}. First, we obtain estimates on the free and perturbed resolvents which hold uniformly with respect to $\delta$.

\begin{figure}
\begin{center}
\begin{picture}(0,0)%
\includegraphics{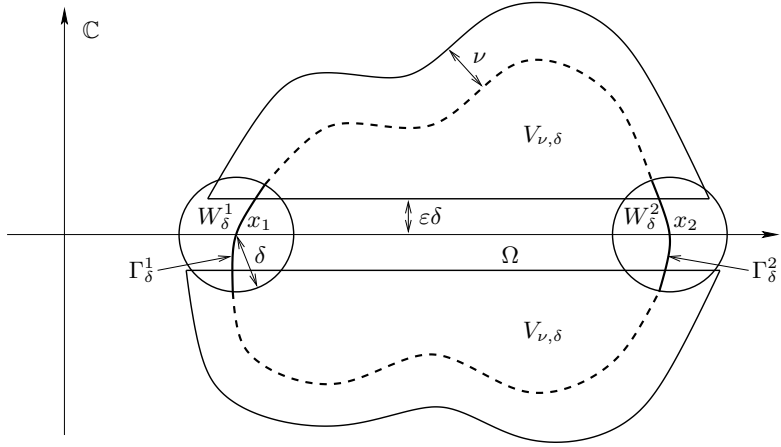}%
\end{picture}%
\setlength{\unitlength}{1184sp}%
\begingroup\makeatletter\ifx\SetFigFont\undefined%
\gdef\SetFigFont#1#2#3#4#5{%
  \reset@font\fontsize{#1}{#2pt}%
  \fontfamily{#3}\fontseries{#4}\fontshape{#5}%
  \selectfont}%
\fi\endgroup%
\begin{picture}(16244,9280)(-2421,-9546)
\put(151,-6061){\makebox(0,0)[lb]{\smash{{\SetFigFont{9}{10.8}{\rmdefault}{\mddefault}{\updefault}$\Gamma_{\delta}^{1}$}}}}
\put(7351,-1561){\makebox(0,0)[lb]{\smash{{\SetFigFont{9}{10.8}{\rmdefault}{\mddefault}{\updefault}$\nu$}}}}
\put(8401,-3211){\makebox(0,0)[lb]{\smash{{\SetFigFont{9}{10.8}{\rmdefault}{\mddefault}{\updefault}$V_{\nu , \delta}$}}}}
\put(8401,-7336){\makebox(0,0)[lb]{\smash{{\SetFigFont{9}{10.8}{\rmdefault}{\mddefault}{\updefault}$V_{\nu , \delta}$}}}}
\put(2626,-5011){\makebox(0,0)[lb]{\smash{{\SetFigFont{9}{10.8}{\rmdefault}{\mddefault}{\updefault}$x_{1}$}}}}
\put(11551,-5011){\makebox(0,0)[lb]{\smash{{\SetFigFont{9}{10.8}{\rmdefault}{\mddefault}{\updefault}$x_{2}$}}}}
\put(1576,-4936){\makebox(0,0)[lb]{\smash{{\SetFigFont{9}{10.8}{\rmdefault}{\mddefault}{\updefault}$W_{\delta}^{1}$}}}}
\put(10501,-4936){\makebox(0,0)[lb]{\smash{{\SetFigFont{9}{10.8}{\rmdefault}{\mddefault}{\updefault}$W_{\delta}^{2}$}}}}
\put(2776,-5761){\makebox(0,0)[lb]{\smash{{\SetFigFont{9}{10.8}{\rmdefault}{\mddefault}{\updefault}$\delta$}}}}
\put(-824,-1036){\makebox(0,0)[lb]{\smash{{\SetFigFont{9}{10.8}{\rmdefault}{\mddefault}{\updefault}$\C$}}}}
\put(7951,-5761){\makebox(0,0)[lb]{\smash{{\SetFigFont{9}{10.8}{\rmdefault}{\mddefault}{\updefault}$\Omega$}}}}
\put(6226,-4936){\makebox(0,0)[lb]{\smash{{\SetFigFont{9}{10.8}{\rmdefault}{\mddefault}{\updefault}$\varepsilon \delta$}}}}
\put(13201,-6061){\makebox(0,0)[lb]{\smash{{\SetFigFont{9}{10.8}{\rmdefault}{\mddefault}{\updefault}$\Gamma_{\delta}^{2}$}}}}
\end{picture}%
\caption{The domains $V_{\nu , \delta}$ and $W_{\delta} = W_{\delta}^{1} \cup W_{\delta}^{2}$, the set $\partial \Omega \cap \R = \{ x_{1} , x_{2} \}$ and the paths $\Gamma_{\delta} = \Gamma_{\delta}^{1} \cup \Gamma_{\delta}^{2}$ with $\Gamma_{\delta}^{j} = \partial \Omega \cap B ( x_{j} , \delta )$.} \label{f2}
\end{center}
\end{figure}

\begin{lemma}\sl \label{a4}
There exists $C >0$ such that, for all $0 < \delta \leq \delta_{0}$, there exists $s ( \delta ) > 0$ such that, for all $0 < s < s ( \delta )$, we have
\begin{equation} \label{a2}
\Big\Vert \Big(I - \frac{A(s z)-K_\delta(s)}{s z}\Big)^{-1} \Big\Vert < \frac{C}{\delta + \vert \im z \vert} ,
\end{equation}
uniformly for $z \in V_{\nu , \delta} \cup W_{\delta /2}$ and
\begin{equation} \label{a3}
\Big\Vert \Big(I - \frac{A(s z)}{s z}\Big)^{-1} \Big\Vert < \frac{C}{\delta + \vert \im z \vert} ,
\end{equation}
uniformly for $z \in V_{\nu , \delta}$.
\end{lemma}

\begin{proof}
We begin by proving \eqref{a2}. First, we make the following decomposition
\begin{equation} \label{a5}
I - \frac{A(s z)-K_\delta(s)}{s z} = \Big(I - \frac{A(s z)-A(0)}{s z} \Big(I - \frac{A(0)-K_\delta(s)}{s z}\Big)^{-1}\Big) \Big(I -\frac{A(0)-K_\delta(s)}{s z}\Big) .
\end{equation}
From the definition \eqref{a21} of $K_\delta(s)$ and the spectral theorem, we get
\begin{equation} \label{a6}
\Big\Vert \Big(I -\frac{A(0)-K_\delta(s)}{s z}\Big)^{-1} \Big\Vert \leq \frac{C}{\delta + \vert \im z \vert} ,
\end{equation}
uniformly for $0 < \delta \leq \delta_{0}$, $s \in ]0 ,1]$ and $z \in V_{\nu , \delta} \cup W_{\delta /2}$.

We will now prove that, for $\delta$ fixed,
\begin{equation} \label{a8}
\slim_{s \to 0} \Big( \overline{\Pi}_{0} \Big(I -\frac{A(0)-K_\delta(s)}{s z}\Big)^{-1} \Big)^{*} = 0 ,
\end{equation}
uniformly for $z \in V_{\nu , \delta} \cup W_{\delta /2}$. For $\alpha >0$ fixed, according to the spectral theorem, there exists $M>0$ such that
\begin{equation*}
\Big\Vert \one_{\{ \vert \lambda \vert \geq M s \}} (A(0)) \Big(I - \frac{A(0) - K_\delta(s)}{s z}\Big)^{-1} \Big\Vert < \alpha ,
\end{equation*}
uniformly with respect to $s \in]0,1]$ and $z \in V_{\nu , \delta} \cup W_{\delta /2}$. On the other hand, the    projection $\one_{ \{0 < \vert \lambda \vert < M s \}} ( A (0) )$ tends strongly to $0$ as $s$ tends to $0$. Then, using \eqref{a6}, we deduce that, for $\delta$ fixed,
\begin{equation*}
\slim_{s \to 0} \Big( \one_{ \{ 0 < \vert \lambda \vert < M s \}} ( A (0) ) \Big( I - \frac{A(0) - K_\delta(s)}{s z} \Big)^{-1} \Big)^{*} = 0 ,
\end{equation*}
uniformly with respect to $z \in V_{\nu , \delta} \cup W_{\delta /2}$. The two last equations imply \eqref{a8}.

Since $A$ is holomorphic near $0$, there exists a holomorphic  operator-valued function $R_2$ such that
\begin{equation*}
\frac{A(s z)-A(0)}{s z}=A^\prime(0) + s z R_2(s z) .
\end{equation*}
Then
\begin{align}
I - \frac{A(s z)-A(0)}{s z} \Big(I - {}& \frac{A(0)-K_\delta(s)}{s z}\Big)^{-1}   \nonumber \\
&= I-A^\prime(0){\Pi}_0- \big( A^\prime(0)\overline{\Pi}_{0}+ s z R_2(s z) \big) \Big(I -\frac{A(0)-K_\delta(s)}{s z}\Big)^{-1}.   \label{a25}
\end{align}
Exploiting \eqref{a6}, \eqref{a8}, $s z R_2(s z) = \cO (s)$, and the compactness of
$A^\prime (0)$, we deduce that, for $\delta$
fixed, the norm of the last term at the right hand side of \eqref{a25}  tends to $0$ as $s \to 0$, uniformly with respect to $z \in V_{\nu , \delta} \cup W_{\delta /2}$. At last, $I - A^\prime(0) \Pi_0$ is invertible by assumption. Then there exists $C>0$ such that, for all $0 < \delta \leq \delta_{0}$, we can choose $s ( \delta ) > 0$ sufficiently small such that
\begin{equation} \label{a7}
\Big\Vert \Big( I - \frac{A(s z)-A(0)}{s z} \Big(I - \frac{A(0)-K_\delta(s)}{s z}\Big)^{-1} \Big)^{-1} \Big\Vert < C ,
\end{equation}
uniformly for $0 < s < s ( \delta )$ and $z \in V_{\nu , \delta} \cup W_{\delta /2}$. Then, \eqref{a2} follows from \eqref{a5}, \eqref{a6} and \eqref{a7}.

The proof of \eqref{a3} is similar. Instead of \eqref{a6}, we use
\begin{equation*}
\Big\Vert \Big(I -\frac{A(0)}{s z}\Big)^{-1} \Big\Vert \leq \frac{s \vert z \vert}{s \vert \im z \vert} \leq \frac{1}{\delta + \vert \im z \vert} ,
\end{equation*}
uniformly for $0 < \delta \leq \delta_{0}$, $s \in ]0 ,1]$ and $z \in V_{\nu , \delta}$.
\end{proof}

\Subsection{Reduction of the problem}

By the definition of the multiplicity of the characteristic values \eqref{c17} and of the index of an operator \eqref{c20}, we have the following

\begin{lemma}\sl \label{a20}
Assume that $s \Omega \subset \D$ and that there are  no characteristic values of $I - A (z) / z$ on $s \partial \Omega$. Then,
\begin{equation*}
\cN ( s \Omega ) = \Ind_{\partial \Omega} \Big( I - \frac{A ( s z )}{s z} \Big) .
\end{equation*}
\end{lemma}

Note that \eqref{a3} implies Corollary \ref{a22} $i)$. Therefore, since the  characteristic values form a discrete set, the assumptions of Lemma \ref{a20} are satisfied for almost all $s$ small enough. Moreover, from the statement of Theorem \ref{a17}, it is enough to prove it for almost all $s$ small enough. Then, we  may always assume in the  sequel that the assumptions of Lemma \ref{a20} are satisfied.

Now, we define
\begin{align}
f_{\delta} (s,z) &= \det \Big( \Big( I - \frac{A ( s z )}{s z} \Big) \Big(I - \frac{A (s z) - K_{\delta} (s)}{s z} \Big)^{-1} \Big)   \nonumber \\
&= \det \Big( I - \frac{K_{\delta} (s)}{s z} \Big(I - \frac{A (s z) - K_{\delta} (s)}{s z} \Big)^{-1} \Big) ,  \label{a27}
\end{align}
the relative determinant which is well defined for $z \in V_{\nu ,
\delta} \cup W_{\delta /2}$ by Lemma \ref{a4} and the finiteness
of ${\rm rank}\,K_{\delta} (s)$.

\begin{lemma}\sl \label{a23}
For all $0 < \delta \leq \delta_{0}$, there exists $s ( \delta ) > 0$ such that, for all $0 < s < s ( \delta )$,
\begin{equation*}
\Ind_{\partial \Omega} \Big( I - \frac{A(s z)}{s z} \Big) = \Ind_{\partial \Omega} \Big( I - \frac{A(s z) - K_{\delta} (s)}{s z} \Big) + \ind_{\partial \Omega} f_{\delta} (s , z ) ,
\end{equation*}
where the index of a holomorphic function is defined in \eqref{d1}.
\end{lemma}

\begin{proof}
Remark that all the quantities are well defined on $\partial \Omega$ since we have assumed the hypotheses of Lemma \ref{a20}. We have
\begin{equation*}
I - \frac{A (s z)}{s z} = I - \frac{A (s z) - K_{\delta} (s)}{s z} - \frac{K_{\delta} (s)}{s z} ,
\end{equation*}
and we can then write
\begin{equation*}
I - \frac{A (s z)}{s z} = \Big( I - \frac{K_{\delta} (s)}{s z} \Big(I - \frac{A (s z)- K_{\delta} (s)}{s z} \Big)^{-1} \Big) \Big( I - \frac{A (s z) - K_{\delta} (s)}{s z} \Big) .
\end{equation*}
Thus we deduce the lemma from \eqref{gr4}--\eqref{gr5}.
\end{proof}

\begin{lemma}\sl \label{a24}
For all $0 < \delta \leq \delta_{0}$, there exists $s ( \delta ) > 0$ such that, for all $0 < s < s ( \delta )$,
\begin{equation*}
\Ind_{\partial \Omega} \Big( I - \frac{A(s z) - K_{\delta} (s)}{s z} \Big) = \Ind_{\partial \Omega} \Big( I - \frac{A (0) - K_{\delta} (s)}{s z} \Big) .
\end{equation*}
\end{lemma}

\begin{proof}
First, by using the following decomposition
\begin{equation*}
I - \frac{A (s z) - K_{\delta} (s)}{s z} = \Big(I - \frac{A (s z)- A(0)}{s z} \Big(I - \frac{A(0) - K_{\delta} (s)}{s z} \Big)^{-1} \Big) \Big(I - \frac{A(0) - K_{\delta} (s)}{s z} \Big) ,
\end{equation*}
we have
\begin{align}
\Ind_{\partial \Omega} \Big(I - \frac{A (s z) - K_{\delta} (s)}{s z}\Big) ={}& \Ind_{\partial \Omega} \Big( I - \frac{A (0) - K_{\delta} (s)}{s z} \Big)   \nonumber \\
&+ \Ind_{\partial \Omega} \Big( I - \frac{A (s z) - A (0)}{s z} \Big( I - \frac{A (0) - K_{\delta} (s)}{s z} \Big)^{-1} \Big) .   \label{a26}
\end{align}
From \eqref{a25}, we can write
\begin{align*}
I - \frac{A(s z)-A(0)}{s z} \Big(I - {}& \frac{A(0)-K_\delta(s)}{s z}\Big)^{-1}   \\
&= I-A^\prime(0){\Pi}_0- \big( A^\prime(0)\overline{\Pi}_{0}+ s z R_2(s z) \big) \Big(I -\frac{A(0)-K_\delta(s)}{s z}\Big)^{-1}.
\end{align*}
Moreover the discussion below \eqref{a25} shows that the last term of
the above equality tends to $0$ as $s$ tends to $0$ uniformly for $z
\in \partial \Omega$ with $0 < \delta \leq \delta_{0}$ fixed.
 Then, since  $I - A^{\prime} (0) \Pi_{0}$ is invertible, using the Rouch\'e theorem (see Theorem \ref{a18}), we deduce
\begin{equation*}
\Ind_{\partial \Omega} \Big( I - \frac{A (s z) - A (0)}{s z} \Big( I - \frac{A (0) - K_{\delta} (s)}{s z} \Big)^{-1} \Big) = \Ind_{\partial \Omega} \big( I - A^{\prime} (0) \Pi_{0} \big) = 0 .
\end{equation*}
 Combining with \eqref{a26}, this concludes the proof.
\end{proof}

\begin{proposition}\sl \label{a1}
Let us consider the function $f_\delta$ introduced in \eqref{a27}. For all $\delta > 0$ small enough, there exists $s ( \delta ) > 0$ such that, for all $0 < s < s ( \delta )$,
\begin{equation*}
\ind_{\dom} f_{\delta} ( s , \cdot )= \cO \big( n (s I_{\delta} ) \vert \ln \delta \vert^{2} \big) .
\end{equation*}
\end{proposition}

We now prove Theorem \ref{a17} and postpone the proof of the crucial Proposition \ref{a1} to the next section. Combining Lemma \ref{a20}, Lemma \ref{a23}, Lemma \ref{a24} with Proposition \ref{a1}, we obtain
\begin{equation*}
\cN ( s \Omega ) = \Ind_{\partial \Omega} \Big( I - \frac{A (0) - K_{\delta} (s)}{s z} \Big) + \cO \big( n (s I_{\delta} ) \vert \ln \delta \vert^{2} \big) .
\end{equation*}
Let $\widetilde{A}_0 := A(0) - K_{\delta} (s) = A (0) \one_{\R \setminus s I_{\delta}} (A(0))$ and  $\widetilde{I}_{\delta} := I_{\delta}\cap \Omega$. Then Theorem \ref{a17} follows from
\begin{align*}
\Ind_{\partial \Omega} \Big( I - \frac{A(0) - K_{\delta} (s)}{s z}
\Big) &= \frac{1}{2 i \pi} \tr \int_{\partial \Omega} \frac{ \widetilde{A}_0}{s z} \Big( z - \frac{\widetilde{A}_{0}}{s} \Big)^{-1} d z = \frac{1}{2 i \pi} \tr \int_{\partial \Omega} \Big( z - \frac{\widetilde{A}_{0}}{s} \Big)^{-1} d z  \\
&= \tr \one_{s J } ( \widetilde{A}_0 )= \tr \one_{s J \setminus s {\widetilde{I}}_{\delta}} ( A(0) )= n ( s J ) - n ( s \widetilde{I}_{\delta} )  \\
&= n ( s J ) + \cO \big( n (s I_{\delta} )  \big).
\end{align*}

In the same way, Remark \ref{c2} follows from the fact that Proposition \ref{a1} holds uniformly with respect to $a_{-} \leq a \leq a_{+}$ and $b_{-} \leq b \leq b_{+}$.

\Subsection{Proof of Proposition \ref{a1}}

We first obtain a factorization of the determinant $f_{\delta} (s,z)$. This idea, due to J. Sj\"{o}strand \cite{Sj97_01}, comes from the study of the semiclassical resonances. Since $\partial \Omega$ is transverse to the real axis, the intersection $\partial \Omega \cap \R$ is a finite set of real numbers $x_{j}$ with $j=1 , \ldots ,J$. Setting $W_{\delta}^{j} : = B ( x_{j} , \delta )$, we have
\begin{equation*}
W_{\delta} = \bigcup_{j=1}^{J} W_{\delta}^{j} .
\end{equation*}
Note that this union is disjoint for $\delta$ small enough (see Figure \ref{f2}).

\begin{lemma}\sl \label{a12}
There exists $C >0$ such that, for all $\delta > 0$ small enough, there exists $s ( \delta ) > 0$ such that, for all $0 < s < s ( \delta )$,

$i)$ for all $j$, the function $f_{\delta} (s,z)$ can be written in $W_{\delta /4}^{j}$ as
\begin{equation*}
f_\delta(s,z) = \prod_{k = 1}^{N_{\delta} (s)} \frac{(z -z_{k}^{\delta} (s))}{\delta} e^{g_{\delta} (s,z)} ,
\end{equation*}
where $z_{k}^{\delta} (s) \in W_{\delta /2}^{j}$, $z \longmapsto g_{\delta} (s,z)$ is holomorphic in $W_{\delta /4}^{j}$ with
\begin{equation} \label{gr1}
N_{\delta} (s) \leq C \vert \ln \delta \vert n ( s I_{\delta} ),
\end{equation}
and
\begin{equation*}
\big\vert g_{\delta}^{\prime} (s,z) \big\vert \leq
C \frac{\vert \ln \delta \vert}{\delta} n ( s I_{\delta} ) , \qquad z \in W_{\delta /4}^{j}.
\end{equation*}

$ii)$ the function $f_{\delta} (s,z)$ has no zeroes  in $V_{\nu /2 , 2 \delta}$ and, in this set,
\begin{equation*}
f_\delta(s,z) = e^{g_{\delta} (s,z)}
 \end{equation*}
with $z \longmapsto g_{\delta} (s,z)$  holomorphic in  $V_{\nu /2
  , 2 \delta}$ and satisfying $\big\vert g_{\delta}^{\prime}
  (s,z) \big\vert \leq C \frac{\< \ln \vert \im z \vert \>}{\vert \im z \vert} n ( s I_{\delta} )$ .
\end{lemma}

To prove Lemma \ref{a12} $i)$, we will apply a complex-analysis result obtained by J. Sj\"{o}strand in the context of the investigation of resonance distribution:

\begin{theorem}[see \cite{Sj97_01}, \cite{Sj01_01}]\sl \label{a28}
Let  $U$ be a simply connected domain of $\C$ with $U \cap \{ \im
\lambda > 0\} \neq \emptyset$. Let $F: U \longrightarrow \C$ be a holomorphic function such that for some $M \geq 1$ we have
\begin{equation*}
\begin{aligned}
&\vert F (\lambda ) \vert \leq e^{M}, &&  \qquad \lambda \in U ,   \\
&\vert F (\lambda ) \vert \geq e^{- M}, && \qquad   \lambda \in U \cap \{ \im \lambda > 0 \} .
\end{aligned}
\end{equation*}
Then, for any $\widetilde{U} \Subset U$, there exists a constant
$C_{\widetilde{U}, U}$ independent of $F$ and $M$, and a
holomorphic function $g : \widetilde{U} \longrightarrow \C$ such
that
\begin{equation*}
  F(\lambda) = \prod_{k =1}^{N } (\lambda - \lambda_{k} ) e^{g (\lambda)}, \qquad  \lambda \in \widetilde{U},
\end{equation*}
where the $\lambda_{k}$'s are  zeroes of $F$ in $U$, $N\leq
C_{\widetilde{U}, U} M$, and $\vert g^{\prime} (\lambda) \vert \leq C_{\widetilde{U}, U} M$ for $\lambda \in \widetilde{U}$.
\end{theorem}

\begin{proof}[Proof of Lemma \ref{a12}]
Since $\cK (z) := \frac{K_\delta(s)}{s z} \Big(I - \frac{A(s z)-K_\delta(s)}{s z}\Big)^{-1}$ is a finite rank operator, we have
\begin{equation*}
f_{\delta} (s,z) = \det \big( I - \cK (z) \big) = \prod_{j=1}^{\rank \cK (z)} ( 1 - \lambda_{j} (z) ) ,
\end{equation*}
where $\lambda_{j} (z)$ are the eigenvalues of $\cK (z)$. On the other hand, the definition of the operator $K_\delta(s) = A(0) {\one}_{I_\delta} ( \frac{A(0)}{s} )$ in \eqref{a21}, and Lemma \ref{a4} yield
\begin{equation*}
\rank \cK (z) = n ( s I_{\delta} ), \qquad \Vert \cK (z) \Vert < \frac{C}{\delta + \vert \im z \vert}, \qquad z \in V_{\nu , \delta} \cup W_{\delta /2} .
\end{equation*}
Then, there exists $C >0$ such that, for all $\delta > 0$ small enough and then $s$ small enough,
\begin{align}
\vert f_{\delta} (s,z) \vert &\leq \big( 1 + \Vert \cK (z) \Vert \big)^{\rank \cK (z)} \nonumber \\
&\leq e^{C n ( s I_\delta ) \vert \ln ( \delta + \vert \im z \vert ) \vert} , \label{a9}
\end{align}
uniformly for $z \in V_{\nu , \delta} \cup W_{\delta /2}$. On the other hand, for $z \in V_{\nu , \delta}$,
\begin{equation*}
f_\delta(s,z)^{-1} = \det \Big( I + \frac{K_\delta(s)}{s z} \Big(I - \frac{A (s z)}{s z} \Big)^{-1} \Big) .
\end{equation*}
Applying once again Lemma \ref{a4}, the previous argument gives
\begin{equation} \label{a10}
\vert f_{\delta} (s,z) \vert \geq e^{- C n ( s I_\delta ) \vert \ln ( \delta + \vert \im z \vert ) \vert} ,
\end{equation}
for $z \in V_{\nu , \delta}$.

Now we apply Theorem \ref{a28} to the function $\lambda \longmapsto F ( \lambda , s, \delta ) :=  f_{\delta} (s , x_{j} + \delta \lambda )$ which is holomorphic in $B (0 ,1/2 )$. Estimates \eqref{a9}--\eqref{a10} give
\begin{equation*}
\begin{aligned}
&\vert F ( \lambda , s, \delta ) \vert \leq e^{C n ( s I_\delta ) \vert \ln \delta \vert}, && \qquad \lambda \in B (0 , 1/2) ,   \\
&\vert F ( \lambda , s, \delta ) \vert \geq e^{-C n ( s I_\delta ) \vert \ln \delta \vert}, && \qquad   \lambda \in B (0 , 1/2 ) \cap \{ \im z > \varepsilon \} .
\end{aligned}
\end{equation*}
Then, Theorem \ref{a28} yields, for all $\lambda \in B ( 0 , 1/4 )$,
\begin{equation*}
f_{\delta} (s , x_{j} + \delta \lambda ) = \prod_{k =1}^{N_{\delta} ( s )} ( \lambda - \lambda_{k}^{\delta} ( s ) ) e^{g_{\delta}  (s , \lambda )} ,
\end{equation*}
with $\lambda_{k}^{\delta} (s) \in B ( 0 , 1/2)$, $N_{\delta} ( s ) \leq \widetilde{C} n ( s I_\delta ) \vert \ln \delta \vert$ and $\vert g^{\prime}_{\delta}  (s , \lambda ) \vert \leq \widetilde{C} n ( s I_\delta ) \vert \ln \delta \vert$. Changing of variable $z = x_{j} + \delta \lambda$, we obtain Lemma
\ref{a12} $i)$.

According to (\ref{a10}), for sufficiently small $s$,  $f_{\delta} (s,z)$ has no zeroes in $V_{\nu , \delta}$. Then there exists $g_{\delta} (s,z)$ holomorphic with
respect to $z \in V_{\nu ,\delta}$ such that
\begin{equation*}
f_\delta(s,z) = e^{g_{\delta} (s,z)} .
\end{equation*}
For $z \in  V_{\nu/2 ,2\delta}$ consider the function
\begin{equation*}
F: \lambda \longmapsto f_{\delta} \Big( s , z +  \lambda \frac{\vert \im z \vert}{4} \Big) , \qquad \lambda \in B (0,2) .
\end{equation*}
Since $F$ has no zeroes in $B(0,2)$, the combination of \eqref{a9}--\eqref{a10} with Theorem \ref{a28}, yields
\begin{equation*}
\frac{\vert \im z \vert}{4} \Big\vert g_\delta^{\prime} \Big( s, z +  \lambda \frac{\vert \im z \vert}{4} \Big) \Big\vert \leq \widetilde{C} n ( s I_\delta ) \vert \ln  \vert \im z \vert \vert , \qquad \lambda
\in B(0,1),
\end{equation*}
with $\widetilde{C}$ independent of $\lambda$, $z$, $s$ and $\delta$. Thus we obtain part  $ii)$ of Lemma \ref{a12}.
\end{proof}

We can now prove Proposition \ref{a1}.  For $j = 1 , \ldots ,J$, set
\begin{equation*}
\Gamma_{\delta}^{j} := \partial \Omega \cap B (x_{j} , \delta ) \qquad \text{ and } \qquad \Gamma_{\delta} := \bigcup_{j = 1}^{J} \Gamma_{\delta}^{j} .
\end{equation*}
For $\delta$ small enough, the $\Gamma_{\delta}^{j}$'s are segments of
size $\delta$ (see Figure \ref{f2}). Moreover, from the assumptions on
$\Omega$, there exists $\varepsilon > 0$ such that $\im z > 2 \varepsilon \delta$ for all $z \in \partial \Omega \setminus
\Gamma_{\delta /4}$. Let us write
\begin{align}
\ind_{\dom} f_{\delta} ( s , z ) &= \frac{1}{2 i \pi} \int_{\dom} \frac{f_{\delta}^{\prime} (s , z)}{f_{\delta} ( s , z )} d z  \nonumber \\
&= \frac{1}{2 i \pi} \int_{\partial \Omega \setminus \Gamma_{\delta /4}} \frac{f_{\delta}^{\prime} (s , z)}{f_{\delta} ( s , z )} d z + \frac{1}{2 i \pi} \int_{\Gamma_{\delta /4}} \frac{f_{\delta}^{\prime} (s , z)}{f_{\delta} ( s , z )} d z . \label{a11}
\end{align}
From Lemma \ref{a12} $ii)$, the first term of \eqref{a11} satisfies
\begin{align}
\Big\vert \frac{1}{2 i \pi} \int_{\partial \Omega \setminus \Gamma_{\delta /4}} \frac{f_{\delta}^{\prime} (s , z)}{f_{\delta} ( s , z )} d z \Big\vert &\leq \frac{1}{2 \pi} \int_{\partial \Omega \setminus \Gamma_{\delta /4}} \big\vert g_{\delta}^{\prime} (s,z) \big\vert \vert d z \vert   \nonumber \\
&\leq \frac{C}{2 \pi} n ( s I_{\delta} ) \int_{\partial \Omega \setminus \Gamma_{\delta /4}} \frac{\< \ln \vert \im z \vert \>}{\vert \im z \vert} \vert d z \vert  \nonumber \\
&\leq C n ( s I_{\delta} ) \vert \ln \delta \vert^{2} . \label{a13}
\end{align}
Here, we have used the elementary identity  $\int_{[ \delta , 1 ]} \frac{\vert \ln x \vert}{x} \, d x = \half \vert \ln \delta \vert^{2}$.

\begin{figure}
\begin{center}
\begin{picture}(0,0)%
\includegraphics{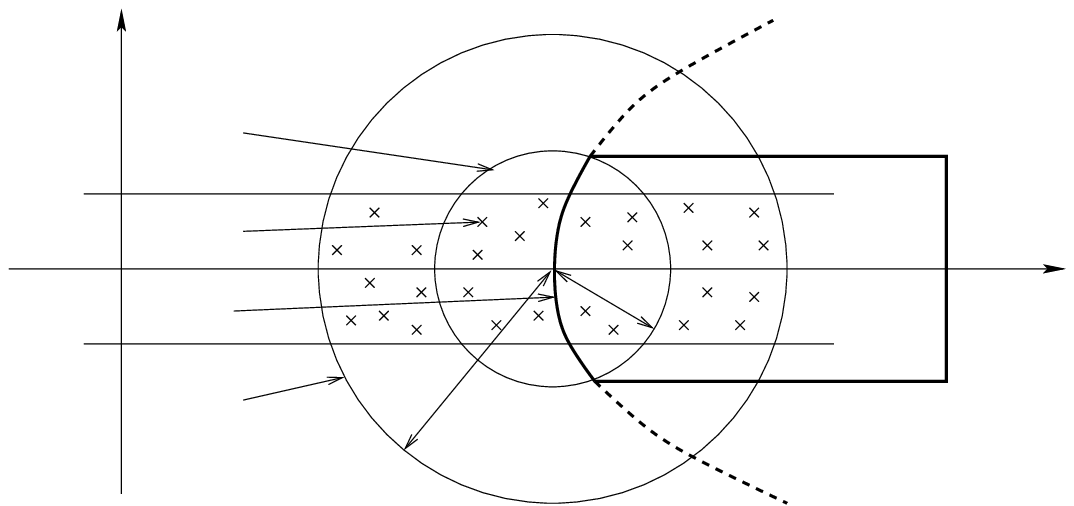}%
\end{picture}%
\setlength{\unitlength}{1184sp}%
\begingroup\makeatletter\ifx\SetFigFont\undefined%
\gdef\SetFigFont#1#2#3#4#5{%
  \reset@font\fontsize{#1}{#2pt}%
  \fontfamily{#3}\fontseries{#4}\fontshape{#5}%
  \selectfont}%
\fi\endgroup%
\begin{picture}(18922,8016)(-4949,-8955)
\put(6601,-5611){\makebox(0,0)[lb]{\smash{{\SetFigFont{9}{10.8}{\rmdefault}{\mddefault}{\updefault}$\delta /4$}}}}
\put(5926,-4936){\makebox(0,0)[lb]{\smash{{\SetFigFont{9}{10.8}{\rmdefault}{\mddefault}{\updefault}$x_{j}$}}}}
\put(12226,-4411){\makebox(0,0)[lb]{\smash{{\SetFigFont{9}{10.8}{\rmdefault}{\mddefault}{\updefault}$\widetilde{\Gamma}_{\delta /4}^{j}$}}}}
\put(10351,-3136){\makebox(0,0)[lb]{\smash{{\SetFigFont{9}{10.8}{\rmdefault}{\mddefault}{\updefault}$C \delta$}}}}
\put(-4934,-4036){\makebox(0,0)[lb]{\smash{{\SetFigFont{9}{10.8}{\rmdefault}{\mddefault}{\updefault}$\im z = o ( \delta )$}}}}
\put(-4934,-6436){\makebox(0,0)[lb]{\smash{{\SetFigFont{9}{10.8}{\rmdefault}{\mddefault}{\updefault}$\im z = - o ( \delta )$}}}}
\put(-449,-4636){\makebox(0,0)[lb]{\smash{{\SetFigFont{9}{10.8}{\rmdefault}{\mddefault}{\updefault}$z^{\delta}_{k} (s)$}}}}
\put(-449,-5911){\makebox(0,0)[lb]{\smash{{\SetFigFont{9}{10.8}{\rmdefault}{\mddefault}{\updefault}$\Gamma_{\delta /4}^{j}$}}}}
\put(-749,-1411){\makebox(0,0)[lb]{\smash{{\SetFigFont{9}{10.8}{\rmdefault}{\mddefault}{\updefault}$\C$}}}}
\put(7876,-1336){\makebox(0,0)[lb]{\smash{{\SetFigFont{9}{10.8}{\rmdefault}{\mddefault}{\updefault}$\partial \Omega$}}}}
\put(-449,-3061){\makebox(0,0)[lb]{\smash{{\SetFigFont{9}{10.8}{\rmdefault}{\mddefault}{\updefault}$W_{\delta /4}^{j}$}}}}
\put(-449,-7303){\makebox(0,0)[lb]{\smash{{\SetFigFont{9}{10.8}{\rmdefault}{\mddefault}{\updefault}$W_{\delta /2}^{j}$}}}}
\put(4051,-7561){\makebox(0,0)[lb]{\smash{{\SetFigFont{9}{10.8}{\rmdefault}{\mddefault}{\updefault}$\delta /2$}}}}
\end{picture}%
\caption{The situation near $x_{j}$ and the set $\widetilde{\Gamma}_{\delta}^{j}$.} \label{f3}
\end{center}
\end{figure}

On the other hand, Lemma \ref{a12} $i)$ implies
\begin{equation} \label{a14}
\frac{1}{2 i \pi} \int_{\Gamma_{\delta /4}^{j}} \frac{f_{\delta}^{\prime} (s , z)}{f_{\delta} ( s , z )} d z = \sum_{k=1}^{N_{\delta} (s)} \frac{1}{2 \pi} \int_{\Gamma_{\delta /4}^{j}} \frac{1}{z - z_{k}^{\delta} (s)} \, d z + \frac{1}{2 \pi} \int_{\Gamma_{\delta /4}^{j}} g_{\delta}^{\prime} (s,z) \, d z .
\end{equation}
Since $\vert \Gamma_{\delta /4}^{j} \vert \leq C \delta$, we get
\begin{equation}
\Big\vert \frac{1}{2 \pi} \int_{\Gamma_{\delta /4}^{j}}
g_{\delta}^{\prime} (s,z) \, d z \Big\vert \leq \frac{C}{2 \pi} n ( s
I_\delta ) \int_{\Gamma_{\delta /4}^{j}} \frac{\vert \ln \delta
  \vert}{\delta} \vert d z \vert \leq C n ( s I_\delta ) \vert \ln \delta \vert .   \label{a15}
\end{equation}
Now, let $\widetilde{\Gamma}_{\delta /4}^{j}$ be as in Figure \ref{f3}. Lemma \ref{a12} yields $z_{k}^{\delta} (s) \in B ( x_{j} , \delta /2 )$. Moreover, thanks to \eqref{a3}, we also have $\im z_{k}^{\delta} (s) = o ( \delta )$ as $s$ tends to $0$. Therefore, the $z_{k}^{\delta} (s)$'s are at distance $\delta$ from $\widetilde{\Gamma}_{\delta /4}^{j}$ for $s$ small enough. Then,
\begin{align}
\Big\vert \frac{1}{2 \pi} \int_{\Gamma_{\delta /4}^{j}} \frac{1}{z - z_{k}^{\delta} (s)} \, d z \Big\vert &= \Big\vert \frac{1}{2 \pi} \int_{\Gamma_{\delta /4}^{j} \cup \widetilde{\Gamma}_{\delta /4}^{j}} \frac{1}{z - z_{k}^{\delta} (s)} \, d z \Big\vert + \Big\vert \frac{1}{2 \pi} \int_{\widetilde{\Gamma}_{\delta /4}^{j}} \frac{1}{z - z_{k}^{\delta} (s)} \, d z \Big\vert  \nonumber \\
&\leq 1 +  \int_{\widetilde{\Gamma}_{\delta /4}^{j}} \frac{C}{\delta} \vert d z \vert \leq C.  \label{gr2}
\end{align}
Combining \eqref{gr2} with \eqref{gr1}, \eqref{a14}, and \eqref{a15}, we get
\begin{equation} \label{a16}
\Big\vert \frac{1}{2 i \pi} \int_{\Gamma_{\delta /4}} \frac{f_{\delta}^{\prime} (s , z)}{f_{\delta} ( s , z )} d z \Big\vert \leq C n ( s I_\delta ) \vert \ln \delta \vert .
\end{equation}
Proposition \ref{a1} now follows from \eqref{a11}, \eqref{a13}, and \eqref{a16}.

\section{Proof of Theorem \ref{a17b} and its corollaries}\label{s4}

In this section we prove Theorem \ref{a17b} applying Theorem
\ref{a17} with appropriate domains $s\Omega$ constructed so that
the associated intervals $I_\delta$ contain a ``small'' number of
eigenvalues of $A(0)$.  First we choose these intervals $I_\delta$
in accordance with the following general result for counting
functions.

\begin{lemma}\sl \label{b1}
There exists $C > 0$ such that, for any $\delta > 0$ small enough and $j \in \N$, there exists $\beta_{j} \in [ j - \frac{1}{4} , j + \frac{1}{4} ]$ such that
\begin{equation*}
n \big( \big[ 2^{-\beta_j} (1 - \delta ) , 2^{- \beta_j} ( 1 + \delta ) \big] \big) \leq C \delta \; n \big( \big[ 2^{- ( j + \frac{1}{4} )} , 2^{- ( j - \frac{1}{4} )} \big] \big) .
\end{equation*}
\end{lemma}

\begin{proof}
First, there exist $\delta_{1} , \varepsilon_{1} > 0$ such that, for all $0 < \delta < \delta_{1}$, one can find disjoint intervals
\begin{equation*}
I_{k} ( \delta ) \subset [ 2^{- \frac{1}{4}} , 2^{\frac{1}{4}} ] ,
\end{equation*}
of the form $[2^{-\beta}(1-\delta),2^{-\beta}(1+\delta)]$ for some $\beta \in [ - \frac{1}{4} , \frac{1}{4} ]$, with integer $k \in [0,\varepsilon_{1} / \delta]$.

Now assume that the assertion of the lemma is not true. Then, for all $C > 0$ and $\delta_{0} > 0$, there exists $0 < \delta < \delta_{0}$ and $j \in \N$ such that, for all $\beta \in [ j - \frac{1}{4} , j + \frac{1}{4} ]$,
\begin{equation*}
n \big( \big[ 2^{-\beta}(1-\delta) , 2^{-\beta}(1+\delta) \big] \big) > C \delta \; n \big( \big[ 2^{- ( j + \frac{1}{4} )} , 2^{- ( j - \frac{1}{4} )} \big] \big) .
\end{equation*}
We choose $C = 2 / \varepsilon_{1}$ and $\delta_{0} = \delta_{1}$. Using the intervals $I_{k} ( \delta )$ constructed previously, we get
\begin{align}
n \big( \big[ 2^{- ( j + \frac{1}{4} )} , 2^{- ( j - \frac{1}{4} )} \big] \big) &\geq n \bigg( \bigcup_{k=0}^{\varepsilon_{1} / \delta} 2^{-j} I_{k} ( \delta ) \bigg) = \sum_{k=0}^{\varepsilon_{1} / \delta} n \big( 2^{-j} I_{k} ( \delta ) \big)   \nonumber \\
&> \frac{\varepsilon_{1}}{\delta} C \delta \; n \big( \big[ 2^{- ( j + \frac{1}{4} )} , 2^{- ( j - \frac{1}{4} )} \big] \big) = 2 \; n \big( \big[ 2^{- ( j + \frac{1}{4} )} , 2^{- ( j - \frac{1}{4} )} \big] \big) ,
\end{align}
which gives a contradiction.
\end{proof}

Combining this lemma with Theorem \ref{a17}, we prove Theorem \ref{a17b}.

\begin{proof}[Proof of Theorem \ref{a17b}]
We consider the sequence $(\beta_j)_j$ constructed in Lemma \ref{b1}. For $\delta > 0$ small enough, let $r ( \delta ) > 0$ be such that Theorem \ref{a17} and Remark \ref{c2} with $a_{-} = a_{+} = 1$, $b_{-} = 2^{\frac{1}{2}}$ and $b_{+} = 2^{\frac{9}{4}}$ hold true for $0 < r < r ( \delta )$. In the sequel, $M ( \delta )$ will denote the smallest integer for which $2^{- \beta_{M ( \delta )}} < r ( \delta )$, and $N ( r)$, $0 < r < r ( \delta )$, will denote  the unique integer such that $2 r < 2^{- N (r)} \leq 4 r$.

By the disjoint decomposition
\begin{equation*}
\s_{\theta}(r,1)=\s_{\theta} \big( r , 2^{-\beta_{N (r)}} \big) \bigcup \bigcup_{j=M ( \delta )}^{N (r) -1}
\, \s_{\theta}(2^{-\beta_{j+1}},2^{-\beta_j}) \bigcup \s_{\theta} \big( 2^{-\beta_{M ( \delta ) }},1 \big) ,
\end{equation*}
we have
\begin{equation} \label{b2}
\cN ( \s_{\theta}(r,1) ) = \cN \big( \s_{\theta} \big(r,2^{-\beta_{N (r)}} \big) \big) + \sum_{j= M ( \delta )}^{N (r) -1} \cN \big( \s_{\theta} \big( 2^{-\beta_{j+1}},2^{-\beta_j} \big) \big) + \cN \big( \s_{\theta} \big( 2^{-\beta_{M ( \delta )}} , 1 \big) \big) .
\end{equation}
By construction of $\beta_j$, we have $2^{-\beta_{N (r)}} \in r ] 2^{\frac{3}{4}} , 2^{\frac{9}{4}} ]$ and
\begin{equation*}
\s_{\theta} \big( 2^{-\beta_{j+1}},2^{-\beta_j} \big) = 2^{-\beta_{j+1}} \s_{\theta} \big( 1,2^{\beta_{j+1}-\beta_j} \big) , \qquad \s_{\theta} \big( 1,2^{\frac12} \big) \subset \s_{\theta} \big( 1,2^{\beta_{j+1}-\beta_j} \big) \subset \s_{\theta} \big( 1,2^{\frac32} \big) .
\end{equation*}
Then from Theorem \ref{a17}, Remark \ref{c2}, and Lemma \ref{b1}, we get
\begin{align}
\cN \big( \s_{\theta} \big( r,2^{-\beta_{N ( r )}} \big) \big) &= n \big( \big[ r , 2^{-\beta_{N ( r )}} \big[ \big) + \cO \Big( \vert \ln \delta \vert^{2} \Big( n ( [ r(1-\delta) , r(1+\delta ) ] )     \nonumber  \\
&\qquad \qquad \qquad \qquad \qquad + n \big( \big[ 2^{-\beta_{N ( r )}} (1-\delta) , 2^{-\beta_{N ( r )}} (1+\delta) \big] \big) \Big) \Big)   \nonumber  \\
&= n \big( \big[ r , 2^{-\beta_{N ( r )}} \big[ \big) + \cO \big( \vert \ln \delta \vert^{2} \big) n ( [ r(1-\delta) , r(1+\delta ) ] )   \nonumber  \\
&\qquad \qquad \qquad \qquad \qquad + \cO \big( \delta \vert \ln \delta \vert^{2} \big) n \big( \big[ 2^{- N ( r ) - \frac{1}{4}} , 2^{- N ( r ) + \frac{1}{4}} \big] \big)   ,  \label{b3}
\end{align}
and, for $M ( \delta ) \leq j \leq N(r) -1$,
\begin{align}
\cN \big( \s_{\theta} \big( 2^{-\beta_{j+1}} , 2^{-\beta_j} \big) \big) &= n \big( \big[ 2^{-\beta_{j+1}} , 2^{-\beta_j} \big[ \big) + \cO \Big( \vert \ln \delta \vert^{2} \Big( n \big( \big[ 2^{-\beta_j} (1-\delta) , 2^{-\beta_j} (1+\delta) \big] \big)   \nonumber  \\
&\qquad \qquad \qquad \qquad \qquad + n \big( \big[ 2^{-\beta_{j+1}} (1-\delta) , 2^{-\beta_{j+1}} (1+\delta) \big] \big) \Big) \Big)    \nonumber \\
&= n \big( \big[ 2^{-\beta_{j+1}} , 2^{-\beta_j} \big[ \big) + \cO \big( \delta \vert \ln \delta \vert^{2} \big) n \big( \big[ 2^{- j - \frac{1}{4}} , 2^{- j + \frac{1}{4}} \big] \big) \nonumber  \\
&\qquad \qquad \qquad \qquad \qquad + \cO \big( \delta \vert \ln \delta \vert^{2} \big) n \big( \big[ 2^{- j - 1 - \frac{1}{4}} , 2^{- j - 1 + \frac{1}{4}} \big] \big)  .
\end{align}
Moreover, we can write
\begin{equation} \label{b4}
\cN \big( \s_{\theta} \big( 2^{-\beta_{M ( \delta )}} , 1 \big) \big) = n \big( \big[ 2^{-\beta_{M ( \delta )}} , 1 \big] \big) + \cO_{\delta} (1) .
\end{equation}

Combining \eqref{b2} with \eqref{b3}--\eqref{b4}, we deduce
\begin{align}
\cN(\s_{\theta}(r,1)) ={}& n([r,1]) + \cO \big( \vert \ln \delta \vert^{2} \big) n ( [ r(1-\delta) , r(1+\delta ) ] )    \nonumber \\
&+ \cO \big( \delta \vert \ln \delta \vert^{2} \big) \sum_{j=M ( \delta )}^{N (r)} n \big( \big[ 2^{- j - \frac{1}{4}} , 2^{- j + \frac{1}{4}} \big] \big) + \cO_{\delta} (1)  \nonumber  \\
={}& n([r,1]) \big( 1 + \cO \big( \delta \vert \ln \delta \vert^{2} \big) \big) + \cO \big( \vert \ln \delta \vert^{2} \big) n ( [ r(1-\delta) , r(1+\delta ) ] ) + \cO_{\delta} (1) ,   \label{b5}
\end{align}
since we have
\begin{equation*}
\bigcup_{j = M ( \delta )}^{N (r)}  \big[ 2^{- j - \frac{1}{4}} , 2^{- j + \frac{1}{4}} \big] \subset [r , 1 ],
\end{equation*}
the union on the left hand side being disjoint. This concludes the proof of Theorem \ref{a17b}.
\end{proof}

In order to prove Corollary \ref{a17c}, we need the following

\begin{lemma}\sl \label{b6}
Let $\Psi : ]0,1[ \longrightarrow \R$ be a non-increasing function such that $\Psi (r) \geq 1$ and $\Psi(r)=\cO(r^{-\gamma})$, $\gamma>0$, on $] 0,1 [$. Then, there exists $C > 0$ such that, for any $\delta > 0$ small enough and any $\rho > 0$, there exists $0 < r \leq \rho$ satisfying
\begin{equation*}
\Psi (r(1-\delta))-\Psi (r(1+\delta)) \leq C \delta \Psi (r) .
\end{equation*}
\end{lemma}

\begin{proof}
Assume that the result is not true. Then, for all $C , \delta_{1} > 0$, there exists $\rho > 0$ and $0 < \delta < \delta_{1}$ such that, for all $0 < r \leq \rho$, we have
\begin{equation*}
\Psi (r(1-\delta))-\Psi (r(1+\delta)) \geq C \delta \Psi (r).
\end{equation*}
Changing the variables $r \longmapsto \sigma r$ with $\sigma=\frac{1 - \delta}{1 + \delta}$, and using the monotonicity and the lower bound of $\Psi$, we get
\begin{equation*}
\Psi ( \sigma r)\geq (1+ C \delta) \Psi (r), \quad r \in ]0,\rho].
\end{equation*}
Then, for any $K\in \N$,
\begin{equation} \label{c3}
\Psi ( \sigma^{K} \rho ) \geq (1+ C \delta)^K \Psi (\rho).
\end{equation}
On the other hand, we have, by assumption,
\begin{equation} \label{c4}
\Psi ( \sigma^{K} \rho ) \leq \cO ( \sigma^{- \gamma K} ) .
\end{equation}
Since \eqref{c3} and \eqref{c4} hold (uniformly) for all $K \in \N$, we deduce
\begin{equation*}
\ln ( 1 + C \delta ) \leq \gamma \vert \ln \sigma \vert .
\end{equation*}
Now, letting $\delta_{1}$ (and, hence, $\delta$) tend to $0$, we find that the  Taylor expansion in $\delta$ yields
\begin{equation*}
C \leq 2 \gamma ,
\end{equation*}
for all $C > 0$. We get a contradiction.
\end{proof}

\begin{proof}[Proof of Corollary \ref{a17c}]
We construct the sequence $( r_{k} )_{k}$ the following way. Let $\delta > 0$ be small enough such that
\begin{equation} \label{end1}
\cO \big( \delta \vert \ln \delta \vert^{2} \big) \leq \frac{1}{k} \qquad \text{and} \qquad C \delta \cO \big( \vert \ln \delta \vert^{2} \big) \leq \frac{1}{k} ,
\end{equation}
where the $\cO$'s are the ones appearing in Theorem \ref{a17b} and $C$ is the constant given in Lemma \ref{b6}. Since $n ([r,1]) \to +\infty$  as $r \searrow 0$, one can find $0 < \rho \leq 2^{-k}$ such that
\begin{equation} \label{end2}
\cO_{\delta} (1) \leq \frac{n([ \rho ,1])}{k} .
\end{equation}
Now, applying Lemma \ref{b6} to the function $\Psi(r):= n([r,1])$ with $C$ and $\delta$ as before, we deduce that there exists $r_{k} \leq \rho$ such that
\begin{equation} \label{end3}
n ( [ r_{k} (1-\delta) , r_{k} (1+\delta ) ] ) \leq C \delta \; n ( [ r_{k} , 1 ] ) .
\end{equation}
By $r_k \in ]0,2^{-k}]$, the positive sequence $(r_k)_{k \in \N}$ tends to 0.
Combining  estimates \eqref{end1}--\eqref{end3} with Theorem \ref{a17b}, we find that
\begin{equation*}
\big\vert \cN(\s_{\theta}(r_{k},1)) - n([r_{k},1]) \big\vert \leq \frac{3}{k} n([r_{k},1]) .
\end{equation*}
which implies \eqref{end4}.
\end{proof}

\begin{proof}[Proof of Corollary \ref{a17d}]
If $n([r,1])= \Phi(r) (1+o(1))$ with
\begin{equation*}
\Phi ( r ( 1 \pm \delta ) ) = \Phi (r) ( 1 + o (1) + \cO ( \delta ) ) ,
\end{equation*}
then $n([r(1-\delta),r(1+\delta)])=n([r,1]) (o(1)+ \cO(\delta))$. In particular, if in addition $\Phi (r)$ tends to infinity, Theorem \ref{a17b} implies that
\begin{equation*}
\cN (\s_{\theta} (r,1 ) ) = \Phi (r) (1 + o(1)) , \qquad r \searrow 0.
\end{equation*}
Thus, Corollary \ref{a17d} follows from the estimates:

$\bullet$ If $\Phi(r) = r^{-\gamma}$, $\gamma >0$, then $\Phi(r(1\pm \delta))=  r^{-\gamma} (1 \pm \delta)^{-\gamma}=\Phi(r) (1 + \cO(\delta))$;

$\bullet$ If $\Phi(r) = \vert \ln r \vert^\gamma$, $\gamma >0$, then $\Phi(r(1\pm \delta))=   \vert \ln r \vert^\gamma \Big( 1 + \frac{\ln (1\pm \delta)}{\ln r}\Big)^\gamma = \Phi(r) (1+o(1))$;

$\bullet$ If $\Phi(r)= \frac{\vert \ln r \vert}{\ln \vert \ln r \vert}$, then $\Phi(r(1\pm \delta))= \frac{\vert \ln r \vert- \ln (1\pm \delta)}{\ln \vert \ln r \vert + \ln \big( 1 + \frac{\ln (1\pm \delta)}{\ln r} \big)}= \Phi(r) (1+o(1))$.
\end{proof}

\section{Application to the counting function of magnetic resonances}\label{s5}

In this section, we apply the results of Section \ref{s2} to the
counting function of resonances of magnetic Schr\"{o}dinger
operators near the Landau levels $2bq$. Let $H_{0}$ be the free Hamiltonian defined in \eqref{gdr0} and \eqref{gdr10}. The selfadjoint operator
$H_0$ is first defined on $C^{\infty}_0(\R^3)$, and then is closed
in $L^2(\R^3)$.  On the domain of $H_0$, we introduce $H : =H_0 +
V$ where $V : \R^3 \longrightarrow \R$ is an appropriate electric
potential. More precisely, we assume that $V$ is Lebesgue
measurable, and satisfies  \eqref{s21} for some $\mper
>0$ and $N>0$.  Under this assumption, the operator $H$ is
selfadjoint with essential spectrum $[0,+\infty[$, and its
resonances near the real axis are defined as the poles of the
meromorphic extension of the resolvent $z \longmapsto (H -
z)^{-1}$ considered as an element of ${\mathcal L} (e^{-N \langle
x_{3} \rangle} L^2 (\R_{\bf x}^{3}) , e^{N \langle x_{3}
\rangle} L^2 (\R_{\bf x}^{3}))$  (for
more details see \cite{BoBrRa07_01}). In Proposition \ref{c8} below, we describe a useful local characterization of the resonances of $H(b,V)$ near a given Landau level $2 b q$, $q \in {\mathbb N}$.
First, near $2 b q$, we parametrize $z$ by
$2 b q + k^2$, and we have

\begin{proposition}[{\cite[Lemma 1]{BoBrRa07_01}}]\sl \label{c6}
For $V$ satisfying \eqref{s21} and $q \in \N$, the operator valued
function
\begin{equation*}
k \longmapsto T_{V,q} (k) : = J \vert V \vert^{\frac{1}{2}} \big( H_0-2b q - k^2 \big)^{-1} \vert V \vert^{\frac{1}{2}}, \qquad J : = \sign V ,
\end{equation*}
defined in $]0,\sqrt{2b}[ e^{i]0, \pi/2[}$, has an analytic
extension to the set $\D \setminus \{ 0 \}$ where $\D : = \{k \in
\C ; \ 0 \leq \vert k \vert < \min ( \sqrt{2 b} , N ) \}$.
\end{proposition}

Then, using the resolvent equation
\begin{equation} \label{c7}
 \big( I - J \vert V \vert^{\frac{1}{2}} (H - z)^{-1} \vert V \vert^{\frac{1}{2}} \big) \big( I + J \vert V \vert^{\frac{1}{2}} (H_0-z)^{-1} \vert V \vert^{\frac{1}{2}} \big) = I ,
\end{equation}
we obtain the  desired characterization of the resonances of $H$:

\begin{proposition}\sl \label{c8}
Under assumption \eqref{s21}, $z_0=2 b q + k_0^2$ is a
resonance of $H$ near $2 b q$, $q \in \N$, if and only if $k_0$ is
a characteristic value of $I+T_{V,q}(\cdot)$, and the multiplicity of  this
resonance coincides with the multiplicity of the characteristic value defined
in Definition \ref{c16}.
\end{proposition}

\begin{proof}
If $m_{\perp} > 2$, Proposition \ref{c8} follows immediately from \cite[Proposition
3]{BoBrRa07_01} and \eqref{c9}.  If
$m_{\perp} \in ]0,2]$, the same proof works using $\det_{p}$ with  $p > 2/ m_{\perp}$.
\end{proof}

In order to formulate our further results, we need the following
notations. Let $p_q$ be the orthogonal projection onto ${\rm ker}(H_{\text{Landau}} - 2bq)$, the Landau
Hamiltonian $H_{\text{Landau}}$ being defined in \eqref{gdr4}. The
operator $p_q$ admits an explicit kernel
\begin{equation*}
{\mathcal P}_{q,b}(\xp,\xp^\prime)=\frac{b}{2\pi} L_q \left( \frac{b \vert \xp - \xp^\prime \vert^2}{2}\right)
\exp \Big( -\frac{b}{4} \big( \vert \xp - \xp^\prime \vert^2 + 2i(x_1x_2^\prime-x_1^\prime x_2) \big) \Big),
\end{equation*}
with $\xp, \xp^\prime \in \R^2$; here $L_q (t) : = \frac{1}{q !}
e^t \frac{d^q ( t^q e^{-t} )}{d t^q}$ are the Laguerre
polynomials. Further, we recall that $I_3$ is the identity
operator in
 $L^2 (\R_{x_3})$. Finally, we denote by $r (z)$  an
operator with integral kernel $\frac12 e^{z \vert
x_3-x^\prime_3\vert}$, $x_3, x_3'\in \R$, depending on the
parameter $z \in \C$.

The following proposition shows that we are in the framework of Section \ref{s2}.

\begin{proposition}[{\cite[Proposition 4]{BoBrRa07_01}}]\sl \label{c15}
Assume that $V$ satisfies \eqref{s21} and fix $q \in \N$. Then for
$k \in {\mathcal D} \setminus \{ 0 \}$, we have
\begin{equation*}
I +T_{V,q} (k) =  I - \frac{A_q ( i k )}{i k} ,
\end{equation*}
where $z \longmapsto A_q(z) \in \cS_\infty(L^2(\R^3))$ is the holomorphic  function given by
\begin{equation} \label{c10}
 A_q(z) = J \vert V \vert^{\frac{1}{2}} p_q \otimes r (z) \vert V \vert^{\frac{1}{2}} - z J \sum_{j \neq q} \vert V \vert^{\frac{1}{2}} ( p_j \otimes I_3 ) \big( D^{2}_{3} + 2 b (j -q ) + z^{2} \big)^{-1} \vert V \vert^{\frac{1}{2}} .
\end{equation}
\end{proposition}

Consequently, the resonances of $H$ near a fixed Landau level $2 b
q$ coincide with the complex number $2 b q + k^2$ where $k$ is a
characteristic value of $\big( I - \frac{A_q ( i k )}{i k} \big)$
and $A_q$ is given by \eqref{c10}. In particular, $A_q(0)$ is the
operator $J \vert V \vert^{\frac{1}{2}} ( p_q \otimes r (0)) \vert
V \vert^{\frac{1}{2}}$ which is selfadjoint as soon as $J$ is
$\pm I$, i.e. for $V$ of definite sign.

Now we assume that $V$  has a definite sign,
 i.e. $\pm V \geq 0$. In order to apply the results of Section \ref{s2}, we want to know when
\begin{equation}\label{c19}
I - A_q^\prime (0) \Pi_0 \text{ is invertible,}
\end{equation}
where, as earlier, $\Pi_0$ is the orthogonal projection on the
kernel of $A_q(0)$. Writing $A_q(0) = \pm L_q^*L_q$ with $L_q:
L^2(\R^3) \longrightarrow L^2(\R^2)$ defined by
\begin{equation} \label{c12}
(L_q f) ( \xp ) : = \frac{1}{\sqrt{2}} \int_{\R_{\bf x}^{3}}
{\mathcal P}_{q,b} ( \xp , \xp^\prime ) \vert V
\vert^{\frac{1}{2}} ( \xp^\prime , x^\prime_3 ) f ( \xp^\prime ,
x^\prime_3 ) \, d \xp^\prime \, d x^\prime_3, \quad \xp \in \R^2,
\end{equation}
we find that ${\rm ker}\,A_q(0) = {\rm ker}\,L_q$.

\begin{remark}\sl \label{c18}
In general, ${\rm ker}\,A_q(0) = {\rm ker}\,L_q$ is not trivial.
Nevertheless, the assumption \eqref{c19} holds  for generic $V$.
More precisely, if the potential $V$ is fixed, there exists a
finite or infinite discrete set ${\mathcal E} = \{ e_{n} \}$ such
that the operator $H_{e} : = H_{0} + e V$ satisfies \eqref{c19}
for all $e \in \R \setminus {\mathcal E}$. The numbers $1/e_n$ are
in fact the  real non vanishing eigenvalues of the compact
operator $A_q^\prime (0) \Pi_0$. To check this, it is enough to
remark that ${\Pi_{0}}_{\vert_{e V}} = {\Pi_{0}}_{\vert_{V}}$ and
${A_q^{\prime} (0)}_{\vert_{e V}} = e {A_q^{\prime}
(0)}_{\vert_{V}}$ for $e \neq 0$. Note also that, for $\vert e
\vert$ small enough, $H_{e}$  satisfies always \eqref{c19}.
\end{remark}

Under these assumptions, we can apply Theorem \ref{a17b} and its
corollaries. Thus, the distribution of the magnetic resonances
near the Landau level is related to the counting function
\begin{equation*}
n_\pm ( s ; A_q (0) ) = n_+ ( s ; L_q^*L_q ) = n_+ ( s ; L_q L_q^* ) = n_+ ( s ; p_q W p_q ) ,
\end{equation*}
where, for a compact selfadjoint operator $T$, we set $n_{\pm} (
s ; T ) = \rank \one_{\pm [ s , + \infty [} (T)$, and $W$ is the multiplication operator by the function
\begin{equation}\label{c13}
W (\xp) : = \frac{1}{2} \int_{\R} \vert V  ( \xp , x_3 ) \vert \, d x_3, \qquad  \xp \in \R^2.
\end{equation}
Let us introduce three types of assumptions for $W$:

{\rm (A1)} $W \in C^1(\R^2)$ satisfies the estimate
\begin{equation*}
W ( \xp ) = w_0 ( \xp / \vert \xp\vert ) \vert \xp \vert^{-
m_\perp} (1 + o(1)), \qquad \vert \xp \vert \rightarrow + \infty ,
\end{equation*}
where $w_0 $ is a continuous function on $\mathbb{S}^1$ which is
non-negative and does not vanish identically, as well as
\begin{equation*}
 \vert \nabla W ( \xp ) \vert \leq C \langle \xp \rangle^{- m_\perp - 1}, \qquad \xp \in \R^2 ,
\end{equation*}
for some constant $C > 0$. Then, by \cite{Ra90_01}, we have
\begin{equation} \label{gr6}
n_+ (r, p_q W p_q ) = C_{\perp} r^{-2 / m_\perp} (1 + o(1)), \qquad r \searrow 0,
\end{equation}
where
\begin{equation*}
C_{\perp} : = \frac{b}{4 \pi} \int_{\mathbb{S}^1} w_0 (t)^{2 / m_\perp} \, d t.
\end{equation*}

{\rm (A2)} There exists $\beta >0$, $\mu >0$ such that
\begin{equation*}
\ln W ( \xp ) = - \mu \vert \xp \vert^{2 \beta} (1 + o(1)), \qquad  \vert \xp \vert \rightarrow + \infty .
\end{equation*}
Then, by \cite{RaWa02_01}, we have
\begin{equation} \label{gr7}
n_+ (r , p_q W p_q) = \varphi_\beta (r) (1+o(1)), \qquad  r \searrow 0 ,
\end{equation}
where, for $0< r \ll 1$,
\begin{equation*}
\varphi_\beta(r) : = \left\{ \begin{aligned}
&\frac{b}{2} \mu^{- \frac{1}{\beta}} \vert \ln r \vert^{\frac{1}{\beta}} &&\text{if } 0 < \beta < 1 ,  \\
&\frac{1}{\ln (1+ 2 \mu/b)} \vert \ln r \vert &&\text{if } \beta = 1 ,    \\
&\frac{\beta}{\beta - 1} ( \ln \vert \ln r \vert )^{-1} \vert \ln r \vert \qquad &&\text{if } \beta > 1 . \\
\end{aligned} \right.
\end{equation*}

{\rm (A3)} The support of $W$ is compact and there exists a
constant $C>0$ such that $W\geq C$ on an non-empty open subset of
$\R^2$. Then, by \cite{RaWa02_01},  we have
\begin{equation} \label{gr8}
n_+ (r, p_q W p_q ) = \varphi_\infty (r) (1+o(1)), \qquad r \searrow 0 ,
\end{equation}
where, for $0< r \ll 1$,
\begin{equation*}
\varphi_\infty (r) : = ( \ln \vert \ln r \vert )^{-1} \vert \ln r \vert .
\end{equation*}
In particular,  $n_+(r, p_q W p_q) \longrightarrow + \infty$
as $r \searrow 0$, provided that $V$ does not vanish identically.

\begin{theorem}\sl \label{c14}
Let $V: \R^3 \longrightarrow \R$ be a Lebesgue measurable function
of definite sign $\pm$ satisfying \eqref{s21} and \eqref{c19}. Let
$0 < r_{0} < \min ( \sqrt{2 b} , N )$ be fixed. Then,

$i)$ The resonances $z_q (k) = 2 b q + k^2$ of $H = H_0 + V$ with $\vert k \vert$ sufficiently small satisfy
\begin{equation*}
\pm \im k \leq 0 , \qquad \re k = o ( \vert k \vert ) .
\end{equation*}

$ii)$ There exists a sequence $(r_\ell)_\ell \in \R$ which tends to $0$ such that
\begin{equation*}
\# \big\{ z = 2 b q + k^2 \in \res ( H ) ; \ r_\ell < \vert k
\vert \leq r_0 \big\} = n_+ ( r_\ell , p_q W p_q ) (1+o(1)),
\qquad \ell \to + \infty .
\end{equation*}

$iii)$ Eventually, if $W$ satisfies {\rm (A1)}, {\rm (A2)} or {\rm (A3)}, then
\begin{equation*}
\# \big\{ z = 2bq + k^2 \in \res (H) ; \ r < \vert k \vert \leq
r_0 \big\} = n_+ ( r , p_q W p_q) ( 1 + o(1) ), \qquad r \searrow
0 ,
\end{equation*}
the asymptotics of $n_+ ( r , p_q W p_q)$ as $r \searrow 0$ being described in \eqref{gr6}--\eqref{gr8}.
\end{theorem}

\begin{remark}\sl
$i)$ Under the assumption {\rm (A1)}, we can also apply Theorem \ref{a17} to obtain asymptotics in small domains.

$ii)$ In \cite{BoBrRa07_01}, we have proved that $H$ has an
infinite number of resonances in a vicinity of $0$ for small
potentials $V$ of definite sign such that $W$, defined in
\eqref{c13}, satisfies, for some $C > 0$,
\begin{equation*}
\ln W ( \xp ) \leq - C \< \xp \>^{2} .
\end{equation*}

$iii)$ These results can be generalized to the case of constant magnetic fields of non full rank $2 r$ in an arbitrary dimension $d$. More precisely, the situation $d - 2 r =1$ is close to the one treated in the present paper. Whereas, if $d - 2 r \geq 3$ is odd, it is expected that there is no accumulation of resonances at the Landau levels since the corresponding $A (z)$ is analytic near these thresholds. The case $d - 2 r$ even is different since the weighted resolvent has a logarithmic singularity at the Landau levels.
\end{remark}

\begin{proof}[Proof of Theorem \ref{c14}]
According to Definition \ref{c8} and Proposition \ref{c15}, in
order to study the resonances $ z_q(k)=2bq + k^2$ of $H$, it is
enough to analyze the characteristic values of $\big( I -
\frac{A_q ( i k )}{i k} \big)$ for $A_q$ given by \eqref{c10}.
Since $\pm A_q (0)$ is non negative, $i)$ is a consequence of
Corollary \ref{a22} with $z = i k$.

From $i)$ we deduce that the resonances $z_q (k) = 2 b q + k^2$
are concentrated in the sector $\mp i \s_{\theta} \cap {\mathcal
D}$ for every $\theta >0$ with $\s_{\theta}$ defined by
\eqref{z1}. In particular, as $r$ tends to $0$, we have:
\begin{align*}
\# \big\{ z = 2bq + k^2 \in \res ( & H) ; \ r_\ell < \vert k \vert \leq r_0 \big\} \\
&= \# \big\{ z = 2bq + k^2 \in \res (H) ; \ \pm i k \in \s_{\theta} (r,r_0) \big\} + \cO (1) .
\end{align*}
Since the non-zero eigenvalues of $\pm A_q(0) = L_q^* L_q$, $L_q$
being defined in \eqref{c12}, coincide with these of $L_q L_q^* =
p_q W p_q$, we have $n ( [ r , r_0 ] ) = n_+ (r, p_q W p_q) +
\cO(1)$. Then, parts $ii)$ and $iii)$ follow from Corollary
\ref{a17c} and Corollary \ref{a17d}.
\end{proof}

\begin{corollary}\sl
For generic potentials $V \geq 0$ satisfying \eqref{s21}, the (embedded) eigenvalues of $H$ form a discrete set.
\end{corollary}

This is a consequence of Remark \ref{c18} and Theorem \ref{c14} $i)$
at each Landau level.  In \cite[Proposition 7]{BoBrRa07_01}, it
is proved that, for small potentials $V \geq 0$ satisfying \eqref{t10}
with $m_{\perp} > 0$ and $m_{3} > 2$, there are no eigenvalues outside
of the Landau levels $2 b \N$. Recall that the setting is very
different for non-positive perturbations. Indeed, for a large class of
non-positive potentials, there is an accumulation of embedded
eigenvalues at each Landau level (see \eqref{y1} and the references
\cite{Ra05_01}, \cite{Ra06_01}).  Note that, using the
Mourre theory, it should be possible to show that the eigenvalues may
accumulate only at the Landau levels (see for instance \cite[Theorem
3.5.3]{GeLa02_01} in a slightly different general context).

\section{Necessity of the assumptions of the main results}\label{s7}

In this section we show that all the assumptions of the Theorem
\ref{a17} and Theorem \ref{a17b} are necessary for the claimed
properties in the sense that these results do not hold if one removes one of their hypotheses (it is perhaps possible to
consider other types of assumptions). An artificial reason would
be easily given by examples for which the characteristic values
are not well defined (for instance when $I- \frac{A(z)}{z}$ is
never invertible). However, the proposition below shows that
there are some more fundamental obstructions to the weakening of
the assumptions.

\begin{proposition}\sl \label{lems7}
Even if the assumptions of Proposition \ref{c21} are satisfied,
the conclusions of the theorems of Section \ref{s2} may be false
if one of the following hypotheses is removed:

$i)$ $A (z)$ is compact-valued;

$ii)$ $A(0)$ is selfadjoint;

$iii)$ $I - A^{\prime}(0) \Pi_0$ is invertible.
\end{proposition}

\begin{proof}
$i)$ If $A (z) = A (0) + A^{\prime} (0) z$ where $A(0)$ is a selfadjoint, compact operator with $\ker A (0)=\{0\}$ and $A^\prime(0) = 2 I$, then the assumptions of Proposition \ref{c21}, Theorem \ref{a17} and Theorem \ref{a17b}, except the compactness of $A (z)$, hold true. In this case, the characteristic values are well defined and given by $( - \lambda_k )_{k \geq 1}$ where $( \lambda_k )_{k \geq 1}$ are the eigenvalues of $A(0)$. Thus, the counting function of the eigenvalues of $A (0)$ and the counting function of the characteristic values of $I - \frac{A (z)}{z}$ are very different.

$ii)$ On ${\mathcal H}= \ell^2(\N)$, let us consider the infinite block
diagonal matrices
\begin{equation*}
A(0)= \diag ( B_0, \ldots, B_k,  \ldots ) \qquad \text{and} \qquad A^{\prime}(0) = \diag ( B^{\prime}_0 , \ldots, B^{\prime}_k , \ldots ) ,
\end{equation*}
where $B_k , B^{\prime}_k$, $k \in \N$, are the following $2 \times 2$ matrices
\begin{equation*}
B_k = \left( \begin{array}{cc} 0 & 0 \\ \alpha_k & 0 \end{array} \right) \qquad \text{and} \qquad B^{\prime}_k = \left( \begin{array}{cc} 0 & \alpha_k \\ 0 & 0 \end{array} \right) ,
\end{equation*}
with $( \alpha_k )_{k \in \N}$ a sequence of real numbers which tends to $0$. Then, the characteristic values of $I- \frac{A(0)}{z} - A^{\prime}(0)$ are the complex numbers $z$ for which one of the matrices
\begin{equation*}
I - \frac{B_{k}}{z} - B^{\prime}_{k} = \left( \begin{array}{cc} 1 & - \alpha_k \\ - \frac{\alpha_k}{z} & 1 \end{array}\right) ,
\end{equation*}
is not invertible. Thus, these characteristic values are the real numbers $\alpha_k^2$, $k \in \N$, while the spectrum of $A(0)$ is reduced to $\{ 0 \}$ since $A (0)$ is nilpotent. Note that $I - A^\prime (0) \Pi_0 = \diag ( I - B^{\prime}_{0} , \ldots , I - B^{\prime}_k , \ldots )$ is invertible. So, we have an example where all the assumptions of Proposition \ref{c21}, Theorem \ref{a17} and Theorem \ref{a17b}, except the selfadjointness of $A (0)$, are met, but  the conclusions of the main theorems of Section \ref{s2} do not hold true.

$iii)$ On $\cH = \C \oplus \ell^{2} ( \N )$, let us consider the  affine function
\begin{equation} \label{end5}
A (z) = A (0) + A^{\prime} (0) z
\end{equation}
 with
\begin{equation} \label{end6}
A (0) = \left( \begin{array}{cc}
0 & 0 \\
0 & \diag ( - \lambda_{k} )
\end{array} \right) \qquad \text{and} \qquad A^{\prime} (0) = \left( \begin{array}{cc}
1 & \alpha \\
{}^{t} \alpha & 0
\end{array} \right) ,
\end{equation}
where $\alpha = ( \alpha_{0} , \ldots , \alpha_{k} , \ldots ) \in \ell^{2} ( \N )$ and $( \lambda_{k} )_{k \in \N}$ goes to $0$. Let
\begin{equation*}
f_{n} (z) = \sum_{k=0}^{n} \alpha_{k}^{2} \Big( \frac{\lambda_{k}}{z} + 1 \Big)^{-1} \qquad \text{and} \qquad f_{\infty} (z) = \sum_{k=0}^{+ \infty} \alpha_{k}^{2} \Big( \frac{\lambda_{k}}{z} + 1 \Big)^{-1} .
\end{equation*}
We construct inductively  two sequences $( \lambda_{k} )_{k \in \N}$ and $( \alpha_{k} )_{k \in \N}$  for which we have
\begin{equation*}
(H)_{n} : \left\{ \begin{aligned}
&0 < \lambda_{n} < \cdots < \lambda_{0} , \\
&\forall 0 \leq k \leq n \quad 0 < \vert \alpha_{k} \vert \leq 2^{- k} \text{ and } \alpha_{k}^{2} \in \R ,  \\
&\forall 0 \leq k \leq n \quad (- 1)^{k} f_{n} (\lambda_{k} ) \geq \frac{\vert \alpha_{k} \vert^{2}}{4} \Big( 1 + \frac{1}{n+1} \Big) ,
\end{aligned} \right.
\end{equation*}
 for all $n \in \N$. One can verify that $(H)_{0}$ holds with $\lambda_{0} = \alpha_{0} = 1$.

\begin{lemma}\sl \label{t2}
If $(H)_{n}$ holds, there exist $\lambda_{n+1}$ and $\alpha_{n+1}$ such that $(H)_{n+1}$ holds.
\end{lemma}

\begin{proof}[Proof of Lemma \ref{t2}]
We choose
\begin{equation*}
\alpha_{n+1} = i^{n+1} \min \Big( 2^{- n -1} , \min_{0 \leq k \leq n} \frac{\vert \alpha_{k} \vert}{2} \Big( \frac{1}{n+1} - \frac{1}{n+2} \Big)^{\frac{1}{2}} \Big) .
\end{equation*}
In particular, $0 < \vert \alpha_{n+1} \vert \leq 2^{- n - 1}$, $\alpha_{n+1}^{2} \in \R$ and
\begin{align*}
(-1)^{k} f_{n+1} ( \lambda_{k} ) &= (-1)^{k} f_{n} ( \lambda_{k} ) + (-1)^{k} \alpha_{n+1}^{2} \Big( \frac{\lambda_{n+1}}{\lambda_{k}} + 1 \Big)^{-1}   \\
&\geq \frac{\vert \alpha_{k} \vert^{2}}{4} \Big( 1 + \frac{1}{n+1} \Big) - \frac{\vert \alpha_{k} \vert^{2}}{4} \Big( \frac{1}{n+1} - \frac{1}{n+2} \Big) \\
&= \frac{\vert \alpha_{k} \vert^{2}}{4} \Big( 1 + \frac{1}{(n+1) + 1} \Big) ,
\end{align*}
for $0 \leq k \leq n$. It remains to evaluate
\begin{align*}
(-1)^{n+1} f_{n+1} ( \lambda_{n+1} ) &= (-1)^{n+1} f_{n} ( \lambda_{n+1} ) + \vert \alpha_{n+1} \vert^{2}
\Big( \frac{\lambda_{n+1}}{\lambda_{n+1}} + 1 \Big)^{-1}  \\
&= \frac{\vert \alpha_{n+1} \vert^{2}}{2} + (-1)^{n+1} f_{n} ( \lambda_{n+1} ) .
\end{align*}
Note that $f_{n} (x) \longrightarrow 0$ as $x \searrow 0$. Then, one can choose
$0 < \lambda_{n+1} < \lambda_{n}$ small enough such that $\vert f_{n} ( \lambda_{n+1} )
\vert \leq \frac{\vert \alpha_{n+1} \vert^{2}}{4}$. So, the previous equation implies
\begin{equation*}
(-1)^{n+1} f_{n+1} ( \lambda_{n+1} ) \geq \frac{\vert \alpha_{n+1} \vert^{2}}{4} ,
\end{equation*}
and $(H)_{n+1}$ holds.
\end{proof}

 Using Lemma \ref{t2}, we can  construct $( \lambda_{k} )_{k \in \N}$ and $( \alpha_{k} )_{k \in \N}$
such that $(H)_{n}$ holds for all $n \in \N$. With such a choice,  the operators $A (0)$ and $A'(0)$ in \eqref{end5}--\eqref{end6} are compact,
 and $A (0)$ is selfadjoint.  However, the operator
\begin{equation*}
I - A^{\prime} (0) \Pi_{0} = \left( \begin{array}{cc}
0 & 0 \\
- {}^{t} \alpha & 1
\end{array} \right) ,
\end{equation*}
is not invertible. On the other hand, $f_{\infty} (z)$ is a well
defined holomorphic function for $z \in \C \setminus  \{ 0 \} \cup
\{ - \lambda_{k} ; \ k \in \N \}$. Moreover, $f_{n}
\longrightarrow f_{\infty}$ uniformly on the compact subset of $\C
\setminus  \{ 0 \} \cup \{ - \lambda_{k} ; \ k \in \N \}$. In
particular, this implies that
\begin{equation} \label{t4}
(- 1)^{k} f_{\infty} (\lambda_{k} ) \geq \frac{\vert \alpha_{k} \vert^{2}}{4} > 0 ,
\end{equation}
for all $k \in \N$.

\begin{lemma}\sl \label{t3}
For $z \in \C \setminus \{ 0 \} \cup \{ - \lambda_{k} ; \ k \in \N \}$, we have
\begin{equation*}
I - \frac{A (z)}{z} \text{ is not invertible} \quad \Longleftrightarrow \quad f_{\infty} (z) = 0 .
\end{equation*}
\end{lemma}

\begin{proof}[Proof of Lemma \ref{t3}]
This result is a direct consequence of the invertibility of $D = 1 + \frac{\diag ( \lambda_{k} )}{z}$ and the identity
\begin{align*}
\left( \begin{array}{cc}
1 & \alpha D^{-1} \\
0 & D^{-1}
\end{array} \right)
\Big( I - \frac{A (z)}{z} \Big) &= \left( \begin{array}{cc}
1 & \alpha D^{-1} \\
0 & D^{-1}
\end{array} \right)
\left( \begin{array}{cc}
0 & - \alpha \\
- {}^{t} \alpha & D
\end{array} \right)   \\
&= \left( \begin{array}{cc}
- \alpha D^{-1} {}^{t} \alpha & 0 \\
- D^{-1} {}^{t} \alpha& 1
\end{array} \right) = \left( \begin{array}{cc}
- f_{\infty} (z) & 0 \\
- D^{-1} {}^{t} \alpha& 1
\end{array} \right) .
\end{align*}
\end{proof}

Combining the previous lemma with \eqref{t4},  we find that $I - \frac{A
(z)}{z}$ is invertible on the $\lambda_{k}$, and then the
assumptions of Proposition \ref{c21} hold. So, the characteristics
values in $\C \setminus \{ 0 \} \cup \{ - \lambda_{k} ; \ k \in \N
\}$ are well defined and coincide with the  zeroes of $f_{\infty}$.
 On the other hand, it follows from $(H)_{n}$ and \eqref{t4}  that $f_{\infty} (x)$ is a
continuous real-valued function on $] 0 , + \infty [$ which
changes its sign between $\lambda_{k+1}$ and $\lambda_{k}$. Then
the intermediate value theorem implies that, for all $k \in \N$,
there exists a characteristic value $x_{k}$ with $0 <
\lambda_{k+1} < x_{k} < \lambda_{k}$. At the same time, the
eigenvalues of $A (0)$ are the $- \lambda_{k}$ which are all
negative.

Summing up, we have constructed an example where
all the assumptions of Proposition \ref{c21}, Theorem \ref{a17} and
Theorem \ref{a17b}, except the invertibility of $I -
A^{\prime} (0) \Pi_{0}$,  are met, but  the conclusions of the main theorems of
Section \ref{s2} do not hold true.
\end{proof}

{\sl Acknowledgments.} J.-F. Bony and V. Bruneau
were partially supported by ANR-08-BLAN-0228. V. Bruneau and G.
Raikov were partially supported by the Chilean Science Foundation
{\em Fondecyt} under Grant 1090467, and {\em N\'ucleo Cient\'ifico
ICM} P07-027-F ``{\em Mathematical Theory of Quantum and Classical
Magnetic Systems''}.

All the three authors thank the Bernoulli
Center, EPFL, Lausanne, for a partial support during the program
``{\it Spectral and dynamical properties of quantum
Hamiltonians}'', January - July 2010.

\bibliographystyle{amsplain}
\providecommand{\bysame}{\leavevmode\hbox to3em{\hrulefill}\thinspace}
\providecommand{\MR}{\relax\ifhmode\unskip\space\fi MR }
\providecommand{\MRhref}[2]{%
  \href{http://www.ams.org/mathscinet-getitem?mr=#1}{#2}
}
\providecommand{\href}[2]{#2}

\end{document}